\documentclass{amsart}
\usepackage{amscd,amssymb,latexsym, amsmath, mathtools}
\usepackage[T1]{fontenc}

\newcommand{\xdownarrow}[1]{%
  {\left\downarrow\vbox to #1{}\right.\kern-\nulldelimiterspace}
}

\input{xypic}

\usepackage{fullpage}

\usepackage{ mathdots }

\usepackage{float}

\usepackage{pifont}
\newcommand{\cmark}{\ding{51}}%
\newcommand{\xmark}{\ding{55}}%
\usepackage[dvipsnames]{xcolor}

\usepackage{enumitem}

\usepackage{tikz}
\usetikzlibrary{decorations.pathreplacing}

\usepackage{hyperref}
\usepackage{color}
\usepackage{microtype}

\newtheorem{thm}[equation]{Theorem}
\newtheorem{cor}[equation]{Corollary}
\newtheorem{lemma}[equation]{Lemma}
\newtheorem{prop}[equation]{Proposition}
\theoremstyle{definition}
\newtheorem{defn}[equation]{Definition}
\theoremstyle{remark}
\newtheorem{remark}[equation]{Remark}
\newtheorem{example}[equation]{Example}

\definecolor{gRed}{HTML}{Fc8d59}
\definecolor{gGreen}{HTML}{2b83ba}
\definecolor{gYellow}{HTML}{FFFFbF}
\definecolor{gRed2}{HTML}{d7191c}

\title{\texorpdfstring{$N_\infty$}{N-infinity}-operads and associahedra}

\author{Scott Balchin}%
\address{Max Planck Institute for Mathematics, Vivatsgasse 7, 53111 Bonn, Germany}
\email{balchin@mpim-bonn.mpg.de}

\author{David Barnes}%
\address{Queen's University Belfast, Mathematical Sciences Research Centre, University Road, Belfast, BT7 1NN. UK.
}
\email{d.barnes@qub.ac.uk}

\author{Constanze Roitzheim}%
\address{University of Kent, School of Mathematics, Statistics and Actuarial Science, Sibson Building, Canterbury, CT2 7FS, UK}
\email{csrr@kent.ac.uk}

\begin{document}
\maketitle

\begin{abstract}
We provide a combinatorial approach to studying the collection of $N_\infty$-operads in $G$-equivariant homotopy theory for $G$ a finite cyclic group of prime power order. In particular, we show that for $G=C_{p^n}$ the natural order on the collection of $N_\infty$-operads is in bijection with the poset structure of the $(n+1)$-associahedron. We further provide a lower bound for the number of possible $N_\infty$-operads for any finite cyclic group $G$. As such, we have reduced an intricate problem in equivariant homotopy theory to a manageable combinatorial problem. 
\end{abstract}

\section{Introduction}

Let $X$ be a topological space equipped with a multiplication $m \colon X \times X \to X$. We say that the multiplication is \emph{homotopy commutative} if the diagram
$$
\xymatrix{X \times X \ar[r]^-{m} \ar[d]|-{\text{twist}} & X \\ X \times X \ar[ur]_-{m}}
$$
commutes up to homotopy, and all higher coherences are satisfied up to homotopy. Homotopy commutativity is neatly encoded by the theory of \emph{$E_\infty$-operads} \cite{MR0420610}. These are those symmetric topological operads such that the space $\mathcal{O}_n$ is $\Sigma_n$-contractible for all $n \geq 0$. As the homotopy type of such operads is determined by the homotopy theory of the underlying spaces $\mathcal{O}_n$, it follows that all $E_\infty$-operads are homotopy equivalent~\cite{MR2016697}. In particular, there is a unique (up to homotopy) way for a space to be homotopy commutative.

We now move to an equivariant setting. We fix a finite group $G$ and consider topological spaces equipped with a $G$-action. In this setting, constructing an appropriate version of homotopy commutativity via  $E_\infty$-operads has its difficulties. For example, there are $G$-operads whose underlying non-equivariant operads are $E_\infty$, but whose derived category of algebras are inequivalent.

To correctly encode homotopy commutativity in the equivariant setting, $N_\infty$-operads were developed. Each $N_\infty$-operad encodes a different notion of homotopy commutativity with respect to the structure of the group. This is in stark contrast to the non-equivariant case where there was a unique way of being homotopy commutative. Understanding different $N_\infty$-operads for a group $G$ is however, in general, challenging. As such, having access to a combinatorial framework in which to study them is of great value. 

Recent work by Blumberg and Hill \cite{MR3406512} led to the conjecture, soon verified by \cite{bonventre,MR3848404,rubin}, that for a group $G$, the data of an $N_\infty$-operad is equivalent to a certain ``indexing system''. 
In Section~\ref{sec:operads} we show that this again is equivalent to a subgraph of the lattice of subgroups satisfying two rules. 
Such a description appears under the name of {\em transfer systems} in~\cite{rubin2}. This opens the door to a more combinatorial approach to studying these operads for a fixed group $G$, which sets the stage for this current paper.

We start with the case of $G$ being a cyclic group $C_{p^n}$ in Section~\ref{sec:primeorder}. A constructive approach leads to our first result that there are $\mathsf{Cat}(n+1)$ many $N_\infty$-operads for $C_{p^{n}}$, where $\mathsf{Cat}(n)$ denotes the $n^{th}$ Catalan number. In particular, there are as many $N_\infty$-operads for $C_{p^n}$ as there are binary trees with $n+2$ leaves.

The relation does not just stop there, though. Binary trees are one way of encoding associahedra (also known as \emph{Tamari lattices} or \emph{Stasheff polytopes}), where a binary tree corresponds to a vertex, and two vertices are related by a directed edge if one tree can be obtained from another by moving a branch to the right. On the other side, the set of all $N_\infty$-operads for $C_{p^n}$ can be ordered by inclusion of the corresponding graphs. We prove that these two posets are in fact isomorphic as posets, i.e., there is an isomorphism between $N_\infty$-operads and binary trees which is order-preserving and order-reflecting.

When moving to a general cyclic group, unfortunately one will quickly find the combinatorics of the $N_\infty$-operads unmanageable. This is due to the fact that in the corresponding graph diagram of an $N_\infty$-operad for $C_{p^{n_1}_1 \cdots p^{n_k}_k}$, the edges not induced from the $C_{p^i}$ become hard to describe. We explain this phenomenon in Section~\ref{sec:generalising} by developing the terms of \emph{pure} and \emph{mixed} $N_\infty$-operads and give a non-trivial lower bound for the number of $N_\infty$-operads for an arbitrary finite cyclic group $G$.

This new approach of $N_\infty$-operads as graph diagrams sheds light on the mysterious nature of the theory of equivariant homotopy commutativity. In particular, we have provided a finite and therefore computable approach to an involved problem.

\subsection*{Acknowledgements}\label{ackref}
We are very grateful for support and hospitality from the University of Kent and from the Isaac Newton Institute for Mathematical Sciences during the programme ``Homotopy harnessing higher structures", which was supported by EPSRC grant EP/R014604/1. We furthermore thank Anna Marie Bohmann and Magdalena K\k{e}dziorek for helpful discussions as well as Jonathan Rubin for insightful comments on an earlier version of this paper.

\section{A brief tour of the theory of \texorpdfstring{$N_\infty$}{N-infinity}-operads}\label{sec:operads}

We shall assume that the reader is somewhat familiar with $G$-equivariant homotopy theory in the sense of~May~\cite{MR1413302}. We shall assume that $G$ is a finite group throughout. Our objects of interest, $N_\infty$-operads, are a special class of $G$-operad, whence we begin our exposition.

\begin{defn}
A \emph{$G$-operad} $\mathcal{O}$ is a symmetric operad in $G$-spaces. That is, we have a sequence of $(G \times \Sigma_n)$-spaces $\mathcal{O}_n$ for all $n \geqslant 0$ such that
\begin{enumerate}[align=left]
\item there is a $G$-fixed identity element $1 \in \mathcal{O}_1$,
\item there are $G$-equivariant composition maps
$$\mathcal{O}_k \times \mathcal{O}_{n_1} \times \cdots \times \mathcal{O}_{n_k} \to \mathcal{O}_{n_1+\cdots + n_k}$$
which satisfy the usual compatibility conditions with each other and the symmetric group actions.
\end{enumerate}
\end{defn}

A certain subclass of $G$-operads, known as $N_\infty$-operads, is used to describe different levels of homotopy commutativity in genuine $G$-equivariant stable homotopy theory, see Blumberg and Hill~\cite{MR3406512}. That is, they are a generalization of $E_\infty$-operads to the equivariant setting. Recall that for a group $G$ a \emph{family} $\mathcal{F}$ is a collection of subgroups which is closed under passage to subgroups and conjugacy. A \emph{universal space} for a family $\mathcal{F}$ is a $G$-space $E\mathcal{F}$ such that for all subgroups $H$ we have
  $$
    (E\mathcal{F})^H \simeq \left\{\begin{array}{lr}
        \ast & H \in \mathcal{F} \\
        \emptyset & H \not\in \mathcal{F}
        \end{array}\right\}.
  $$

\begin{defn}\label{def:op}
An \emph{$N_\infty$-operad} is a $G$-operad $\mathcal{O}$ such that
\begin{enumerate}[align=left]
\item the space $\mathcal{O}_0$ is $G$-contractible,
\item the action of $\Sigma_n$ on $\mathcal{O}_n$ is free,
\item $\mathcal{O}_n$ is a universal space for a family $\mathcal{F}_n(\mathcal{O})$ of subgroups of $G \times \Sigma_n$ which contains all subgroups of the form $H \times \{1\}$ for $H \leqslant G$.
\end{enumerate}
We will denote by $N_\infty(G)$ the collection of all (homotopy classes of) $N_\infty$-operads for a given group $G$.
\end{defn}

Although Definition~\ref{def:op} is perfectly good for theoretical purposes, 
we need a more computationally exploitable definition of $N_\infty$-operads. 
We first introduce the intermediary notion of indexing systems.

\begin{defn}
A \emph{categorical coefficient system} is a contravariant functor $\underline{\mathcal{C}} \colon O_G^{op} \to \textbf{Cat}$ from the orbit category of $G$ to the category of small categories. Such a coefficient system is called \emph{symmetric monoidal} if it takes values in symmetric monoidal categories and strong monoidal functors. We are particularly interested in the coefficient system $\mathcal{S} et$ with disjoint union which sends a subgroup $H$ to the category $\textbf{Set}^H$ of $H$-sets, where $\textbf{Set}$ denotes the category of finite sets. A sub-symmetric coefficient $\underline{\mathcal{C}}$ of $\mathcal{S} et$ is said to be an \emph{indexing system} if it is closed under sub-objects and self-induction (i.e., $T \in \underline{\mathcal{C}}(K) $ and $H/K \in \underline{\mathcal{C}}(H)$ implies that $H \times_K T \in \underline{\mathcal{C}}(H)$).
\end{defn}

The following result was first conjectured in Blumberg and Hill~\cite{MR3406512} and has subsequently been proven to hold in three independent articles by Bonventre and Pereira; Guti\'{e}rrez and White; and Rubin. The result uses the existence of a model structure on the category of  $N_\infty$-operads whose weak equivalences are those maps which (at level $n$) induce weak homotopy equivalences after taking $\Gamma$-fixed points for all $\Gamma \subseteq G \times \Sigma_n$.

\begin{prop}[{\cite{bonventre,MR3848404,rubin}}]
The homotopy category of $N_\infty$-operads is equivalent to the poset category of indexing systems.
\end{prop}

An algebra for an $N_\infty$-operad has structure above and beyond being a $G$-spectrum 
whose underlying non-equivariant spectrum is an $E_\infty$-algebra. This additional structure was 
fundamental to Hill, Hopkins and Ravenel~\cite{MR3505179}.
 
\begin{thm}[{\cite[Theorems 4.13 and 4.14]{MR3773736}}]\label{thm:normsandoperads}
Let $G$ be a finite group and $\mathcal{O}$ an $N_\infty$-operad. 
If $R$ is an $\mathcal{O}$-algebra in $G$-spectra, then the 
Mackey functor $\underline{\pi}_0(R)$ is an $\mathcal{O}$-Tambara functor. 
\end{thm}

That is, each $\underline{\pi}_0(R)(G/H) \coloneqq \pi_0^H(R)$ is a commutative ring, 
the restriction maps are monoidal and the transfer maps satisfy the Frobenius relations. 
Furthermore, if $H/K \in \underline{\mathcal{C}}(H)$, for 
$\underline{\mathcal{C}}$ the indexing system determined by $\mathcal{O}$,
there is a multiplicative (but not usually additive) norm map 
\[
N_K^H \colon \pi_0^K(R) \longrightarrow \pi_0^H(R).
\]
The norm maps satisfy Frobenius-type relations describing their interaction with 
addition and the restriction and transfer maps. 
An $\mathcal{O}$-Tambara functor is also known as an \emph{incomplete Tambara functor}. 
In particular, the maps $N_H^H$ are always present and are the identity maps.
If one has a norm map for each pair of subgroups $K \leqslant H$ of $G$, then the homotopy groups 
of an $\mathcal{O}$-algebra are a Tambara functor in the original sense of 
\cite{MR1209937}.

We now compare the notion of indexing systems to \emph{transfer systems} from~\cite{rubin2}. This notion was also independently discovered by the authors.

\begin{lemma}[{\cite[\S 6]{rubin2}}]
An indexing system determines, and is determined by, a set $\mathcal{F}_H$ for each $H \leqslant G$ consisting of subgroups $K$ of $H$, written
as $H/K$, satisfying the following axioms
\begin{itemize}[align=left]
\item[(Identity)] $H/H \in \mathcal{F}_H$.
\item[(Conjugation)] $H/K \in \mathcal{F}_H$ implies $gHg^{-1} / gKg^{-1} \in \mathcal{F}_{gH g^{-1}}$.
\item[(Restriction)] $H/K \in \mathcal{F}_H$ implies  $M/(M \cap K) \in \mathcal{F}_M$ for all $M \leqslant H$.
\item[(Composition)] $H/K \in \mathcal{F}_H$ and $K/L \in \mathcal{F}_K$ implies $H/L \in \mathcal{F}_H$.
\end{itemize}
We call this data a \emph{transfer system}.
\end{lemma}

We rewrite this definition into a form more directly useful for our purposes.
\begin{defn}
Given a transfer system $\mathcal{F}_H$, define a set (of abstract symbols)
\[
\{
N_K^H \mid H/K \in \mathcal{F}_H, \ \ H \leqslant G
\}
\]
and call the symbol $N_K^H$ a \emph{norm map} from $K$ to $H$. 
\end{defn}

\begin{cor}\label{operad-struc}
Let $G$ be a finite group.
Up to homotopy, an $N_\infty$-operad for $G$ is the data of a set of
norm maps $X = \{ N_K^H \}_{1 \leqslant K < H \leqslant G}$ satisfying the following rules (and all conjugates thereof).
\begin{itemize}[align=left]
\item[(Restriction)] If $N_K^H \in X$ and $M<H$, then $N_{K \cap M}^M \in X$.
\item[(Composition)] If $N_L^K \in X$ and $N_K^H \in X$, then $N_L^H \in X$.
\end{itemize}
In particular, homotopy classes of $N_\infty$-operads can be described as certain subgraphs of the lattice of subgroups of $G$.
\end{cor}

\begin{proof}
Up to homotopy, an $N_\infty$-operad $\mathcal{O}$ uniquely specifies a transfer system $\mathcal{F}$.
We know that $H/K \in \mathcal{F}_H$ implies
\[
M/(M \cap K) \in \mathcal{F}_M.
\]
In terms of norm maps this is precisely the restriction rule.
The second axiom of a transfer system says that
$H/K \in \mathcal{F}_H$ and $K/L \in \mathcal{F}_K$ implies $H/L \in \mathcal{F}_H$.
In terms of norm maps this is precisely the composition rule.

For the converse, we construct a transfer system from a set of norm maps $X$ satisfying the two axioms
of the statement. We define $\mathcal{F}_H$ as the set of $H/K$ such that $N_K^H \in X$ and $H/H$. 
One can check this defines a transfer system and hence a homotopy class of 
$N_\infty$-operads.
\end{proof}

\begin{remark}
Our notation has been chosen so that the Mackey functor of an algebra $R$ over an $N_\infty$-operad $\mathcal{O}$ has 
multiplicative norm maps $N_K^H \colon \pi_0^K(R) \longrightarrow \pi_0^H(R)$ 
whenever $N_K^H$ is in the (abstract) set of norm maps. 
\end{remark}

\begin{remark}\label{rem:comp}
Given a set of norm maps $\{N_K^H\}$ which do not satisfy the rules of  Corollary~\ref{operad-struc}, then one can find a minimal set $X$ containing $\{N_K^H\}$ which does satisfy the rules. We refer to this as the \emph{completion} of the set $\{N_K^H\}$ to a transfer system $X$. See also~\cite[Theorem A.2]{rubin2}.

\end{remark}

This results leads to the following corollary, which motivates the results in this paper, namely, that for a finite group $G$, it makes sense to attempt to enumerate the number of $N_\infty$-operad structures, and to understand the associated poset structure.

\begin{cor}
Let $G$ be a finite group. Then the number of $N_\infty$-operads for $G$ is finite. Moreover, the set $N_\infty(G)$ admits a canonical poset structure given by inclusions of sets of the corresponding transfer systems.
\end{cor}

\section{The case \texorpdfstring{$G = C_{p^n}$}{G=Cpn}}\label{sec:primeorder}

We will begin with the case of cyclic groups of the form $C_{p^n}$. We note that the choice of $p$ here is arbitrary as the subgroup lattices of $C_{p^n}$ and $C_{q^n}$ are both isomorphic to the poset $\underline{n} = \{0 < 1 < \dots < n\}$. To ease the notation we shall denote by $N_i^j$ the norm map $N_{C_{p^i}}^{C_{p^j}}$ for $i \leqslant j$.

Before we continue to the theoretics, let us manually compute the first handful of values of $|N_\infty(C_{p^n})|$. The purpose of this is two-fold. Firstly it will give the reader an idea of how such computations are done, and second, for the avid integer sequence fan, these examples will suggest the general form for the sequence $\{|N_\infty(C_{p^n})|\}_{n \in \mathbb{N}}$. Note that we will not write the identity norm maps $N_i^i$ and shall only consider the non--trivial norm maps.

\begin{example}
The case of $G = C_{p^0}$ is trivial. That is, there are no choices of non-trivial indexing systems to make, and therefore $|N_{\infty}(C_{p^0})| = 1$. This is exactly the fact that for non-equivariant stable homotopy theory, there is only a single notion of commutativity as one may expect. We will write the single (up to homotopy)
$N_\infty$-operad structure as $\{\emptyset\}$ to indicate that there are no non-trivial norm maps.
\end{example}

\begin{example}
The situation for $G = C_p$ is only marginally more involved than the trivial case. Here we have a subgroup lattice $\{C_{p^0} < C_{p^1}\}$. Therefore the only choice to make is if we wish to include the non-trivial norm map $N_0^1$ or not. Therefore there are two $N_\infty$-operad structures (up to homotopy), namely $\{\emptyset\}$ and $\{N_0^1\}$.
\end{example}

\begin{example}
We shall now look at $G=C_{p^2}$. This is the first case where we need to take care of the rules appearing in Corollary~\ref{operad-struc}. As always, we have the trivial $N_\infty$-operad $\{\emptyset\}$ which we shall write diagrammatically as

\vspace{5mm}

\begin{figure}[h!]
\centering
\begin{tikzpicture}
\node (0) at (2,0) {$C_{p^0}$};
\node (1) at (4,0) {$C_{p^1}$};
\node (2) at (6,0) {$C_{p^2}$};
\node at (1.5,0) {\Huge{(}};
\node at (6.5,0) {\Huge{)}};
\node at (7.4,0) {$=$};
\node[label=right:{$\{\emptyset\}$ .}] at (8,0) {};
\end{tikzpicture}
\end{figure}

\vspace{5mm}

At the other extreme, we could add in all of the norm maps. One can easily check the conditions to see that this will always be a valid $N_\infty$-operad. We shall display this $N_\infty$-operad as

\vspace{5mm}

\begin{figure}[h!]
  \begin{tikzpicture}[->, node distance=2cm, auto]
\node (3) at (6.000000,0) {$C_{p^2}$};
\node (2) at (4.000000,0) {$C_{p^1}$};
\node (1) at (2.000000,0) {$C_{p^0}$};
\draw (1) to (2);
\draw (2) to (3);
\draw (1) edge[bend left=30, looseness=0.5, ->] (3);
\node at (1.5,0) {\Huge{(}};
\node at (6.5,0) {\Huge{)}};
\node at (7.4,0) {$=$};
 \node[label=right:{$\{N_{0}^{1},N_1^2,N_0^2\}$}] at (8.0,0) {};
 \end{tikzpicture}
 \end{figure}

\vspace{5mm}

\noindent where an arrow from $C_{p^{i}}$ to $C_{p^{j}}$ indicates the existence of the norm map $N_i^j$ for $i < j$.

 The technical part then, of course, is to identify what other $N_\infty$-operads can appear in-between these two extremes. There are $2^3$ different possibilities to try (indeed, there are three different norm maps which we much choose whether to include or not). Instead of investigating all of the remaining cases, we shall just show the failure of the ones that do not have an $N_\infty$-operad structure. Figures~\ref{inv1}, \ref{inv2} and \ref{inv3} give the invalid diagrams.

 \begin{figure}[H]
  \begin{tikzpicture}[->, node distance=2cm, auto,scale=.9]
\node (3) at (6.000000,0) {$C_{p^2}$};
\node (2) at (4.000000,0) {$C_{p^1}$};
\node (1) at (2.000000,0) {$C_{p^0}$};
\draw (1) to (2);
\draw (2) to (3);
 \node[label=right:{$\{N_{0}^{1},N_1^2\}$}] at (8,-0.0) {};
 \node[label=\textcolor{red}{\xmark}] at (0,-0.4) {};
 \node at (1.5,0) {\Huge{(}};
\node at (6.5,0) {\Huge{)}};
\node at (7.4,0) {$=$};
 \end{tikzpicture}
 \caption{This diagram is not valid as it violates the composition rule of Corollary~\ref{operad-struc}.  If we were to complete this diagram in the sense of Remark~\ref{rem:comp} to get a valid $N_\infty$-operad then we would need to add in the norm map $N_0^2$, and we get the previous operad.}\label{inv1}
 \end{figure}

 \begin{figure}[h!]
  \begin{tikzpicture}[->, node distance=2cm, auto,scale=.9]
\node (3) at (6.000000,0) {$C_{p^2}$};
\node (2) at (4.000000,0) {$C_{p^1}$};
\node (1) at (2.000000,0) {$C_{p^0}$};
\draw (1) edge[bend left=30, looseness=0.5, ->] (3);
\node at (1.5,0) {\Huge{(}};
\node at (6.5,0) {\Huge{)}};
\node at (7.4,0) {$=$};
 \node[label=right:{$\{N_0^2\}$}] at (8,0) {};
  \node[label=\textcolor{red}{\xmark}] at (0,-0.4) {};
 \end{tikzpicture}
 \caption{This diagram is not valid as it does not satisfy the restriction rules.  To satisfy the rule we would need to also have the norm map $N_0^1$, and then all of the rules would be satisfied. The resulting operad would be different from the above two.}\label{inv2}
 \end{figure}

 \begin{figure}[h!]
  \begin{tikzpicture}[->, node distance=2cm, auto,scale=.9]]
\node (3) at (6.000000,0) {$C_{p^2}$};
\node (2) at (4.000000,0) {$C_{p^1}$};
\node (1) at (2.000000,0) {$C_{p^0}$};
\draw (2) to (3);
\draw (1) edge[bend left=30, looseness=0.5, ->] (3);
 \node[label=right:{$\{N_1^2,N_0^2\}$}] at (8,-0) {};
 \node[label=\textcolor{red}{\xmark}] at (0,-0.4) {};
 \node at (1.5,0) {\Huge{(}};
\node at (6.5,0) {\Huge{)}};
\node at (7.4,0) {$=$};
 \end{tikzpicture}
 \caption{This is the final invalid diagram, which suffers from the same deficiency as the one above, that is, it does not satisfy the restriction rules.} \label{inv3}
 \end{figure}

Consequently, we can write down the elements of $N_\infty(C_{p^2})$. Note that in particular, $|N_\infty(C_{p^2})| = 5$. We implore the reader to check these for themselves to gain confidence with the rules of Corollary~\ref{operad-struc}. The valid $N_\infty$-operad structures are as follows.

\begin{figure}[h!]
\centering
 \begin{tikzpicture}[->, node distance=2cm, auto,scale=.9]
\node (3) at (6.000000,0) {$C_{p^2}$};
\node (2) at (4.000000,0) {$C_{p^1}$};
\node (1) at (2.000000,0) {$C_{p^0}$};
 \node[label=right:{$\{\emptyset\}$}] at (8,-0) {};
 \node[label=\textcolor{Green}{\cmark}] at (0,-0.4) {};
 \node at (1.5,0) {\Huge{(}};
\node at (6.5,0) {\Huge{)}};
\node at (7.4,0) {$=$};
 \end{tikzpicture}
\end{figure}

\begin{figure}[h!]
 \begin{tikzpicture}[->, node distance=2cm, auto,scale=.9]
\node (3) at (6.000000,0) {$C_{p^2}$};
\node (2) at (4.000000,0) {$C_{p^1}$};
\node (1) at (2.000000,0) {$C_{p^0}$};
\draw (1) to (2);
 \node[label=right:{$\{N_{0}^{1}\}$}] at (8,-0) {};
 \node[label=\textcolor{Green}{\cmark}] at (0,-0.4) {};
 \node at (1.5,0) {\Huge{(}};
\node at (6.5,0) {\Huge{)}};
\node at (7.4,0) {$=$};
 \end{tikzpicture}
 \end{figure}

\begin{figure}[h!]
 \begin{tikzpicture}[->, node distance=2cm, auto,scale=.9]
\node (3) at (6.000000,0) {$C_{p^2}$};
\node (2) at (4.000000,0) {$C_{p^1}$};
\node (1) at (2.000000,0) {$C_{p^0}$};
\draw (1) to (2);
\draw (1) edge[bend left=20, looseness=0.5, ->] (3);
 \node[label=right:{$\{N_{0}^{1},N_0^2\}$}] at (8,-0) {};
 \node[label=\textcolor{Green}{\cmark}] at (0,-0.4) {};
 \node at (1.5,0) {\Huge{(}};
\node at (6.5,0) {\Huge{)}};
\node at (7.4,0) {$=$};
 \end{tikzpicture}
   \end{figure}

\begin{figure}[h!]
  \begin{tikzpicture}[->, node distance=2cm, auto,scale=.9]
\node (3) at (6.000000,0) {$C_{p^2}$};
\node (2) at (4.000000,0) {$C_{p^1}$};
\node (1) at (2.000000,0) {$C_{p^0}$};
\draw (2) to (3);
 \node[label=right:{$\{N_{1}^{2}\}$}] at (8,-0) {};
 \node[label=\textcolor{Green}{\cmark}] at (0,-0.4) {};
 \node at (1.5,0) {\Huge{(}};
\node at (6.5,0) {\Huge{)}};
\node at (7.4,0) {$=$};
 \end{tikzpicture}
  \end{figure}

\begin{figure}[h!]
  \begin{tikzpicture}[->, node distance=2cm, auto,scale=.9]
\node (3) at (6.000000,0) {$C_{p^2}$};
\node (2) at (4.000000,0) {$C_{p^1}$};
\node (1) at (2.000000,0) {$C_{p^0}$};
\draw (1) to (2);
\draw (2) to (3);
\draw (1) edge[bend left=20, looseness=0.5, ->] (3);
 \node[label=right:{$\{N_{0}^{1},N_1^2,N_0^2\}$}] at (8,-0) {};
 \node[label=\textcolor{Green}{\cmark}] at (0,-0.4) {};
 \node at (1.5,0) {\Huge{(}};
\node at (6.5,0) {\Huge{)}};
\node at (7.4,0) {$=$};
 \end{tikzpicture}
 \end{figure}
\end{example}

From our first analysis, we have obtained the integer sequence $1,2,5$ counting the number of $N_\infty$-operads for $C_{p^0}$, $C_{p^1}$ and $C_{p^2}$ respectively. If one were to take the time to check the possibilities for $C_{p^3}$, they would see that there are 14 possibilities. Therefore the examples suggest a relation to the Catalan numbers. The next section will be devoted to recalling the necessary results regarding the Catalan numbers before we prove the first main result, Theorem~\ref{mainthm}, which says that $|N_\infty(C_{p^n})|$ coincides with the $(n+1)$-st Catalan number.

\subsection{A recollection of the Catalan numbers}\label{sec:cat}

The Catalan numbers are a sequence of numbers which regularly appear in enumeration problems. The $n^{th}$ Catalan number, which we denote $\mathsf{Cat}(n)$, is given as
$$\mathsf{Cat}(n) = \dfrac{(2n)!}{(n+1)!n!} \rlap{ .}$$
The first few terms of the sequence are therefore $\mathsf{Cat}(0) = 1$, $\mathsf{Cat}(1) = 1$, $\mathsf{Cat}(2) = 2$, $\mathsf{Cat}(3) = 5$ and $\mathsf{Cat}(4)=14$. There are many surprising and amusing ways to define the Catalan numbers. Let us recall a few.
\begin{itemize}
\item $\mathsf{Cat}(n)$ is the number of valid expressions containing $n$ pairs of parentheses.
\item $\mathsf{Cat}(n)$ is the number of triangulations of a regular $(n+2)$-gon.
\item $\mathsf{Cat}(n)$ is the number of rooted binary trees with $n+1$ leaves.
\end{itemize}
This is but a few of a multitude of descriptions given in Stanley~\cite{MR1676282}. The last interpretation involving binary trees will be our canonical representation. Figure~\ref{trees} gives the corresponding binary trees in the case of $n=2$.

\vspace{5mm}

\begin{figure}[h!]
\centering
\begin{tikzpicture}[scale=0.5]
\node [fill,circle,draw,inner sep = 0pt, outer sep = 0pt, minimum size=2mm] (1) at (5,1) {};
\node [fill,circle,draw,inner sep = 0pt, outer sep = 0pt, minimum size=2mm] (3) at (6,0) {};
\draw (6,-1) edge node {} (3);
\draw (3) edge node {} (8,2);
\draw (1) edge node {} (4,2);
\draw (1) edge node {} (6,2);
\draw (1) edge node {} (3);

\node [fill,circle,draw,inner sep = 0pt, outer sep = 0pt, minimum size=2mm] (11) at (12,1) {};
\node [fill,circle,draw,inner sep = 0pt, outer sep = 0pt, minimum size=2mm] (33) at (11,0) {};
\draw (11,-1) edge node {} (33);
\draw (33) edge node {} (13,2);
\draw (33) edge node {} (9,2);
\draw (11) edge node {} (11,2);
\draw (11) edge node {} (33);
\end{tikzpicture}
\caption{The two binary trees giving the $2^{nd}$ Catalan number.}\label{trees}
\end{figure}

\vspace{5mm}

The following well-known recurrence relation will be fundamental to the proof of the main result in this section.

\begin{lemma}
The Catalan numbers satisfy, and are determined by, the recurrence relation
\begin{align*}
\mathsf{Cat}(0) &=1, \\
\mathsf{Cat}(n+1) &= \sum^n_{i=0} \mathsf{Cat}(i) \mathsf{Cat}(n-i) \text{ for } n \geqslant 0.
\end{align*}
\end{lemma}

\subsection{An operation on \texorpdfstring{$N_\infty$}{N-infinity}-operads}

To facilitate the proof of Theorem \ref{mainthm}, we first introduce a binary operation
$$\odot \colon N_\infty(C_{p^i}) \times N_\infty(C_{p^j}) \to N_\infty(C_{p^{i+j+2}}).$$

To be able to define this function explicitly, we need some auxiliary notation. We consider an $X \in N_\infty(C_{p^i})$ as being described by its finite set of norm maps. Secondly, for formal reasons, we will fix the convention that $N_\infty(C_{p^{-1}})$ is defined to be the set containing only the empty set.

For $X \in N_\infty(C_{p^i})$ and $Y \in N_\infty(C_{p^j})$, we now define $X \odot Y \in N_\infty(C_{p^{i+j+2}})$ to be the $N_\infty$-operad described by the set of norm maps
$$X \odot Y \coloneqq  X \coprod \Sigma^{i+2} Y \coprod \{N_{i+1}^k\}_{i+1 < k \leq i+j+2}.$$
Here, $\Sigma$ is the following notation.
For a norm map $N_{k_1}^{k_2}$, we write $$\Sigma^n N_{k_1}^{k_2} \coloneqq  N_{{k_1+n}}^{{k_2+n}}$$
and let $\Sigma^n Y$ denote the set resulting from applying $\Sigma^n$ to each norm map in $Y$. The symbol $\Sigma$ is used to invoke the idea of a suspension or shift.

Figures~\ref{dotdia}, \ref{dotdia2} and \ref{dotdia3} give a pictorial presentation
of $X \odot Y$. We exclude the norm maps for $X$ and $Y$ from the diagrams for clarity. Note, that in particular, we can see that this operation is not commutative (nor is it associative). 

\begin{figure}[h!]
\centering
\begin{tikzpicture}[scale = 0.75]
\draw [black,fill=gGreen] (0,0) rectangle (5,0.5);
\draw [black] (5.75,0) rectangle (6.25,0.5);
\draw [black,fill=gRed] (7,0) rectangle (9,0.5);

\node [fill,circle,draw,inner sep = 0pt, outer sep = 0pt, minimum size=1.5mm] (from) at (6,.25) {};

\node [fill,circle,draw,inner sep = 0pt, outer sep = 0pt, minimum size=1.5mm] (to1) at (8.75,.25) {};
\node [fill,circle,draw,inner sep = 0pt, outer sep = 0pt, minimum size=1.5mm] (to2) at (8.5,.25) {};

\node [fill,circle,draw,inner sep = 0pt, outer sep = 0pt, minimum size=1.5mm] (to3) at (7.25,.25) {};
\node [fill,circle,draw,inner sep = 0pt, outer sep = 0pt, minimum size=1.5mm] (to4) at (7.5,.25) {};

\node [fill,circle,draw,inner sep = 0pt, outer sep = 0pt, minimum size=1.5mm] at (4.75,.25) {};
\node [fill,circle,draw,inner sep = 0pt, outer sep = 0pt, minimum size=1.5mm] at (4.5,.25) {};

\node [fill,circle,draw,inner sep = 0pt, outer sep = 0pt, minimum size=1.5mm] at (0.25,.25) {};
\node [fill,circle,draw,inner sep = 0pt, outer sep = 0pt, minimum size=1.5mm] at (0.5,.25) {};

\draw[->,bend left=80,thick] (from) to (to3);
\draw[->,bend left=80,thick] (from) to (to4);
\draw[->,bend left=80,thick] (from) to (to1);
\draw[->,bend left=80,thick] (from) to (to2);

\node at (8.025,0.25) {$\cdots$};
\node at (2.525,0.25) {$\cdots$};

\draw [decorate,decoration={brace,amplitude=10pt}]
(5.0,-0.15) -- (0.0,-0.15) node [black,midway,yshift=-0.6cm] {\footnotesize $X \in N_\infty(C_{p^i})$};

\draw [decorate,decoration={brace,amplitude=10pt}]
(9.0,-0.15) -- (7.0,-0.15) node [black,midway,yshift=-0.6cm] {\footnotesize $Y \in N_\infty(C_{p^j})$};

\draw [decorate,decoration={brace,amplitude=10pt}]
(9.0,-1.15) -- (0.0,-1.15) node [black,midway,yshift=-0.6cm] {\footnotesize $X \odot Y \in N_\infty(C_{p^{i+j+2}})$};

\end{tikzpicture}
\caption{The general picture for the operation $X \odot Y$. We highlight that the last vertex in $X$ occurs at position $i$, the pivot point is in spot $i+1$, and the first vertex of the suspension of $Y$ occurs at position $i+2$.}\label{dotdia}
\end{figure}

\begin{figure}[h!]
\centering
\begin{tikzpicture}[scale = 0.75]
\draw [black] (0,0) rectangle (0.5,0.5);
\draw [black,fill=gRed] (1.25,0) rectangle (9,0.5);

\node [fill,circle,draw,inner sep = 0pt, outer sep = 0pt, minimum size=1.5mm] (from) at (.25,.25) {};

\node [fill,circle,draw,inner sep = 0pt, outer sep = 0pt, minimum size=1.5mm] (to1) at (8.75,.25) {};
\node [fill,circle,draw,inner sep = 0pt, outer sep = 0pt, minimum size=1.5mm] (to2) at (8.5,.25) {};

\node [fill,circle,draw,inner sep = 0pt, outer sep = 0pt, minimum size=1.5mm] (to3) at (1.5,.25) {};
\node [fill,circle,draw,inner sep = 0pt, outer sep = 0pt, minimum size=1.5mm] (to4) at (1.75,.25) {};

\draw[->,bend left=60,thick] (from) to (to3);
\draw[->,bend left=60,thick] (from) to (to4);
\draw[->,bend left=50,thick] (from) to (to1);
\draw[->,bend left=50,thick] (from) to (to2);

\node at (5.125,0.25) {$\cdots$};

\draw [decorate,decoration={brace,amplitude=10pt}]
(9.0,-0.15) -- (1.25,-0.15) node [black,midway,yshift=-0.6cm] {\footnotesize $Y \in N_\infty(C_{p^j})$};

\draw [decorate,decoration={brace,amplitude=10pt}]
(9.0,-1.15) -- (0.0,-1.15) node [black,midway,yshift=-0.6cm] {\footnotesize $\emptyset \odot Y \in N_\infty(C_{p^{j+1}})$};

\end{tikzpicture}
\caption{The general picture for the operation $\emptyset \odot Y$.}\label{dotdia2}
\end{figure}

\begin{figure}[h!]
\centering
\begin{tikzpicture}[scale = 0.75]
\draw [black,fill=gGreen] (0,0) rectangle (7.75,0.5);
\draw [black] (8.5,0) rectangle (9,0.5);

\node [fill,circle,draw,inner sep = 0pt, outer sep = 0pt, minimum size=1.5mm] (to1) at (8.75,.25) {};
\node [fill,circle,draw,inner sep = 0pt, outer sep = 0pt, minimum size=1.5mm] (to3) at (7.25,.25) {};
\node [fill,circle,draw,inner sep = 0pt, outer sep = 0pt, minimum size=1.5mm] (to4) at (7.5,.25) {};
\node [fill,circle,draw,inner sep = 0pt, outer sep = 0pt, minimum size=1.5mm] at (0.25,.25) {};
\node [fill,circle,draw,inner sep = 0pt, outer sep = 0pt, minimum size=1.5mm] at (0.5,.25) {};

\node at (3.875,0.25) {$\cdots$};

\draw [decorate,decoration={brace,amplitude=10pt}]
(7.75,-0.15) -- (0.0,-0.15) node [black,midway,yshift=-0.6cm] {\footnotesize $X \in N_\infty(C_{p^i})$};

\draw [decorate,decoration={brace,amplitude=10pt}]
(9.0,-1.15) -- (0.0,-1.15) node [black,midway,yshift=-0.6cm] {\footnotesize $X \odot \emptyset \in N_\infty(C_{p^{i+1}})$};

\end{tikzpicture}
\caption{The general picture for the operation $X \odot \emptyset$.}\label{dotdia3}
\end{figure}

\newpage

Let us give some explicit examples of this construction before we prove that the resulting set of norm maps does indeed give an $N_\infty$-operad as we have claimed.

\begin{example}
Let
$$X = \left( \begin{gathered}  \begin{tikzpicture}[->, node distance=1cm,gGreen]
\node (1) at (2.000000,0) {$C_{p^1}$};
\node (2) at (0.000000,0) {$C_{p^0}$};
 \end{tikzpicture}  \end{gathered} \right) \in N_\infty(C_{p^1}),$$

 $$Y = \left( \begin{gathered}  \begin{tikzpicture}[->, node distance=1cm,gRed]
\node (1) at (2.000000,0) {$C_{p^1}$};
\node (2) at (0.000000,0) {$C_{p^0}$};
\draw (2) to (1);
 \end{tikzpicture}  \end{gathered} \right) \in N_\infty(C_{p^1}).$$

\vspace{10mm}

Then $X \odot Y = \left( \begin{gathered}  \begin{tikzpicture}[->, node distance=1cm]
\node[gGreen] (1) at (2.000000,0) {$C_{p^1}$};
\node[gGreen] (0) at (0.000000,0) {$C_{p^0}$};
\node (2) at (4.000000,0) {$C_{p^2}$};
\node[gRed] (3) at (6.000000,0) {$C_{p^3}$};
\node[gRed] (4) at (8.000000,0) {$C_{p^4}$};
\draw (2) to (3);
\draw[gRed] (3) to (4);
\draw[bend left = 20] (2) to (4);
 \end{tikzpicture}  \end{gathered} \right) \in N_\infty(C_{p^4}),$

\vspace{10mm}

and $Y \odot X = \left( \begin{gathered}  \begin{tikzpicture}[->, node distance=1cm]
\node[gRed] (1) at (2.000000,0) {$C_{p^1}$};
\node[gRed] (0) at (0.000000,0) {$C_{p^0}$};
\node (2) at (4.000000,0) {$C_{p^2}$};
\node[gGreen] (3) at (6.000000,0) {$C_{p^3}$};
\node[gGreen] (4) at (8.000000,0) {$C_{p^4}$};
\draw[gRed] (0) to (1);
\draw (2) to (3);
\draw[bend left = 20] (2) to (4);
 \end{tikzpicture}  \end{gathered} \right) \in N_\infty(C_{p^4}).$
\end{example}

\begin{example}
Let
$$X = \left( \begin{gathered}  \begin{tikzpicture}[->, node distance=1cm,gGreen]
\node (3) at (6.000000,0) {$C_{p^3}$};
\node (2) at (4.000000,0) {$C_{p^2}$};
\node (1) at (2.000000,0) {$C_{p^1}$};
\node (0) at (0.000000,0) {$C_{p^0}$};
\draw (0) to (1);
\draw (2) to (3);
 \end{tikzpicture}  \end{gathered} \right) \in N_\infty(C_{p^3})$$

 $$Y = N_\infty(C_{p^{-1}}) = \emptyset$$

Then $X \odot Y = \left( \begin{gathered}  \begin{tikzpicture}[->, node distance=1cm,gGreen]
\node[black] (4) at (8.000000,0) {$C_{p^4}$};
\node (3) at (6.000000,0) {$C_{p^3}$};
\node (2) at (4.000000,0) {$C_{p^2}$};
\node (1) at (2.000000,0) {$C_{p^1}$};
\node (0) at (0.000000,0) {$C_{p^0}$};
\draw (0) to (1);
\draw (2) to (3);
 \end{tikzpicture}  \end{gathered} \right) \in N_\infty(C_{p^4})$

and $Y \odot X = \left( \begin{gathered}  \begin{tikzpicture}[->, node distance=1cm]
\node[gGreen] (4) at (8.000000,0) {$C_{p^4}$};
\node[gGreen] (3) at (6.000000,0) {$C_{p^3}$};
\node[gGreen] (2) at (4.000000,0) {$C_{p^2}$};
\node[gGreen] (1) at (2.000000,0) {$C_{p^1}$};
\node (0) at (0.000000,0) {$C_{p^0}$};
\draw[gGreen] (1) to (2);
\draw[gGreen] (3) to (4);
\draw (0) to (1);
\draw[bend left = 20] (0) to (2);
\draw[bend left = 20] (0) to (3);
\draw[bend left = 20] (0) to (4);
 \end{tikzpicture}  \end{gathered} \right) \in N_\infty(C_{p^4}).$
\end{example}

\begin{prop}\label{isvalid}
For $X \in N_\infty(C_{p^i})$ and $Y \in N_\infty(C_{p^j})$, $X \odot Y$ satisfies the rules of Corollary~\ref{operad-struc}, and therefore is a valid object in $N_\infty(C_{p^{i+j+2}})$ for $-1 \leqslant i,j$. Moreover, the converse is true, that is, if $X \odot Y \in N_\infty(C_{p^{i+j+2}})$, then it follows that $X$ and $Y$ are both valid $N_\infty$-operads for their respective groups.
\end{prop}

\begin{proof}
We must check that $X \odot Y$ satisfies the restriction and composition conditions. The simplest way to do this is to appeal to Figure~\ref{dotdia}. First of all, note that the the norm maps coming from $X$ are disjoint from the rest of the structure, and as we have assumed that $X$ is a valid $N_\infty$-operad for $G=C_{p^i}$, this part does not need further consideration.

The restriction rule for the remaining norm maps is clear. This rule is satisfied due to the addition of the norm maps $\{N_{i+1}^k\}_{i+1 < k \leq i+j+2}$. The composition rule will be satisfied because $Y$ was chosen to be in $N_\infty(C_{p^j})$, and suspending it to its new position will not affect this.

To see the converse of the statement, take two lattices $X$ and $Y$ of size $i$ and $j$ respectively
such that $X \odot Y \in N_\infty(C_{p^{i+j+2}})$.
We first of all note that $X$ must be an object of $N_\infty(C_{p^i})$. Clearly if $Y$ was not an object in $N_\infty(C_{p^j})$, then neither would its shift. Therefore it only remains to show that the addition of the norm maps $\{N_{i+1}^k\}_{i+1 < k \leq i+j+2}$ has no possibility of invalidating $Y$.
As mentioned above, adding these maps only serves to ensure the
restriction rule is satisfied for the additional point,
hence they cannot turn $Y$ into a invalid diagram.
\end{proof}

\begin{remark}
The $\odot$ operation has an operadic interpretation, as explained to the authors by J. Rubin. Suppose that we have two transfer systems $X$ and $Y$ which realise the operads $\mathcal{O}$ and $\mathcal{P}$ which are $C_{p^i}$ and $C_{p^j}$ $N_\infty$-operads respectively.

The inclusion $C_{p^i} \hookrightarrow C_{p^{i+j+2}}$ gives rise to a left derived induction functor, which when applied to
$\mathcal{O}$ realises $X$ as a $C_{p^{i+j+2}}$ $N_\infty$-operad.
Similarly, the quotient map $\pi \colon C_{p^{i+j+2}} \to C_{p^j}$ gives a left derived restriction functor, which when applied to $\mathcal{P}$ realises $\Sigma^{i+2} Y$.
Finally, there is a little disks operad $\mathcal{D}$ which realises the set of norm maps  $\{N_{i+1}^k\}_{i+1 < k \leq i+j+2}$. The homotopy coproduct of these three operads realises $X \odot Y$.

In particular, this result tells us that it is possible to use a homotopy colimit construction to inductively form the homotopy categories of $N_\infty$-operads for $G = C_{p^n}$. Of course, one hopes for a general result like this for arbitrary $G$, but as we will see in Section~\ref{sec:generalising} the situation becomes extremely complex.
\end{remark}

\subsection{Computing the cardinality of \texorpdfstring{$N_\infty(C_{p^n})$}{N-infinity(Cpn)}}

We now come to the first main result of this paper which gives the link between the set of $N_\infty$-operads for $C_{p^n}$ and the Catalan numbers. We shall prove that the cardinalities of these sets satisfy the defining recurrence relation for the Catalan numbers, and then we show how to construct a bijection between these $N_\infty$-operads and binary trees.

\begin{thm}\label{mainthm}
The cardinalities $|N_\infty(C_{p^n})|$ satisfy the recurrence relation
\begin{align*}
|N_\infty(C_{p^{-1}})| &=1, \\
|N_\infty(C_{p^{n}})| &= \sum^n_{i=0}|N_\infty(C_{p^{i-1}})| |N_\infty(C_{p^{n-i-1}})| \text{ for } n \geqslant 0.
\end{align*}
In particular we have that $|N_\infty(C_{p^{n}})| = \mathsf{Cat}(n+1)$.
\end{thm}

\begin{proof}
To prove this we shall show that every $N_\infty$-operad in $Z \in N_\infty(C_{p^{n}})$ can be written in the form $X \odot Y$ for (unique) $X \in N_\infty(C_{p^{i-1}})$ and $Y \in N_\infty(C_{p^{n-i-1}})$. This fact, along with Proposition~\ref{isvalid} completes the argument.

Suppose that $Z \in N_\infty(C_{p^{n}})$. We let $k \in \mathbb{Z}$ be the minimum integer such that the norm map $N_k^n$ is in $Z$. We have three cases to deal with here, either $k=0$, $0< k < n$ or $k=n$ (i.e., there is no such norm map). We start with the two extreme cases before dealing with the intermediate one.

\begin{itemize}
\item When $k=0$, we construct Z as $X \odot Y$ for $X = \emptyset \in N_\infty(C_{p^{-1}})$, and $Y$ an $N_\infty$-operad for $G=C_{p^{n-1}}$ as in Figure~\ref{dotdia2}.
\item When $k=n$, we construct Z as $X \odot Y$ for $Y = \emptyset \in N_\infty(C_{p^{-1}})$), and $X$ an $N_\infty$-operad for $G=C_{p^{n-1}}$ as in Figure~\ref{dotdia3}.
\item When $0<k<n$, we observe that we have two disjoint parts to $Z$, namely we are able to split off the subgroups $C_{p^i}$ for $0 \leqslant i < k$. Let us denote this part as $X$ (which lives in $N_\infty(C_{p^{k-1}})$, and the remaining part $Z'$. The crucial observation to make now is that $Z'$ looks like $\emptyset \odot Y$ for some $Y \in N_\infty(C_{p^{n-k-1}})$. We therefore conclude that $Z = X \odot Y$ as required.
\end{itemize}
\end{proof}

\begin{cor}
Every $N_\infty$-operad $Z$ for $G=C_{p^n}$ can be decomposed uniquely as $Z = X \odot Y$ for some $N_\infty$-operads $X$ and $Y$.
\end{cor}

\begin{cor}\label{tree-cor}
There is a bijection of sets $\{N_\infty (C_{p^n} )\} \Leftrightarrow \{\text{rooted binary trees with n+2 \text{leaves}} \}$.
\end{cor}

\begin{proof}
This follows immediately from Theorem~\ref{mainthm} and the discussion in \S\ref{sec:cat}, however, it will be beneficial to the next section to spell out exactly how the correspondence works inductively. To the trivial $N_\infty$-operad for $G=C_{p^0}$ we assign the following binary tree.
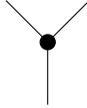
\begin{figure}[h]
\centering
\begin{tikzpicture}[scale=0.55]
\node [fill,circle,draw,inner sep = 0pt, outer sep = 0pt, minimum size=2mm] (3) at (0,1) {};
\draw (0,-.5) edge node {} (3);
\draw (1,2) edge node {} (3);
\draw (-1,2) edge node {} (3);
\end{tikzpicture}
\end{figure}

We will make the convention that $\emptyset$ is the empty tree. Assume that $n > 0$, we know from Theorem~\ref{mainthm} that any $N_\infty$-operad is of the form $X \odot Y$. We then have a binary tree associated to $X$ and a binary tree associated to $Y$, and we can form the binary tree associated to $X \odot Y$ in the following way.
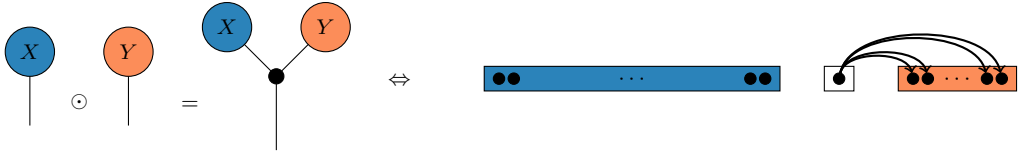
\begin{figure}[h!]
\centering
\begin{tikzpicture}[scale=0.65]
\draw[fill=gGreen] (-5,1.5) circle [radius=0.5] node{$X$};
 \node[label={$\odot$}] at (-4,0) {};
\draw[fill=gRed] (-3,1.5) circle [radius=0.5] node{$Y$};
\draw (-5,1) edge (-5,0);
\draw (-3,1) edge (-3,0);
 \node[label={$=$}] at (-1.75,0) {};

\node [fill,circle,draw,inner sep = 0pt, outer sep = 0pt, minimum size=2mm] (3) at (0,1) {};
\draw (0,-.5) edge node {} (3);
\draw (1,2) edge node {} (3);
\draw (-1,2) edge node {} (3);
\draw[fill=gGreen] (-1,2) circle [radius=0.5] node{$X$};
\draw[fill=gRed] (1,2) circle [radius=0.5] node{$Y$};

\begin{scope}[xshift = 120, yshift = 20, xscale = 1.2]
\draw [black,fill=gGreen] (0,0) rectangle (5,0.5);
\draw [black] (5.75,0) rectangle (6.25,0.5);
\draw [black,fill=gRed] (7,0) rectangle (9,0.5);

\node [fill,circle,draw,inner sep = 0pt, outer sep = 0pt, minimum size=1.5mm] (from) at (6,.25) {};

\node [fill,circle,draw,inner sep = 0pt, outer sep = 0pt, minimum size=1.5mm] (to1) at (8.75,.25) {};
\node [fill,circle,draw,inner sep = 0pt, outer sep = 0pt, minimum size=1.5mm] (to2) at (8.5,.25) {};

\node [fill,circle,draw,inner sep = 0pt, outer sep = 0pt, minimum size=1.5mm] (to3) at (7.25,.25) {};
\node [fill,circle,draw,inner sep = 0pt, outer sep = 0pt, minimum size=1.5mm] (to4) at (7.5,.25) {};

\node [fill,circle,draw,inner sep = 0pt, outer sep = 0pt, minimum size=1.5mm] at (4.75,.25) {};
\node [fill,circle,draw,inner sep = 0pt, outer sep = 0pt, minimum size=1.5mm] at (4.5,.25) {};

\node [fill,circle,draw,inner sep = 0pt, outer sep = 0pt, minimum size=1.5mm] at (0.25,.25) {};
\node [fill,circle,draw,inner sep = 0pt, outer sep = 0pt, minimum size=1.5mm] at (0.5,.25) {};

\draw[->,bend left=80,thick] (from) to (to3);
\draw[->,bend left=80,thick] (from) to (to4);
\draw[->,bend left=80,thick] (from) to (to1);
\draw[->,bend left=80,thick] (from) to (to2);

\node at (8.025,0.25) {$\cdots$};
\node at (2.525,0.25) {$\cdots$};
\end{scope}

 \node[label={$\Leftrightarrow$}] at (2.5,0.5) {};

\end{tikzpicture}
\end{figure}

Following the convention of the empty diagram, we see that
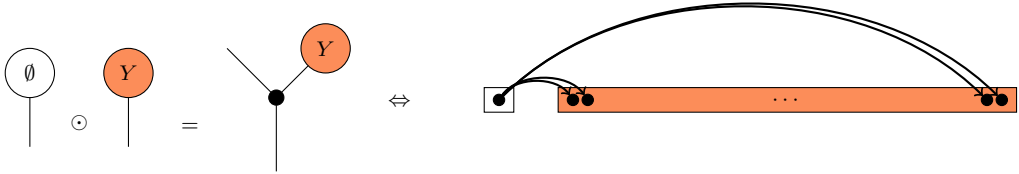
\begin{figure}[h!]
\centering
\begin{tikzpicture}[scale=0.65]
\draw[fill=white] (-5,1.5) circle [radius=0.5] node{$\emptyset$};
 \node[label={$\odot$}] at (-4,0) {};
\draw[fill=gRed] (-3,1.5) circle [radius=0.5] node{$Y$};
\draw (-5,1) edge (-5,0);
\draw (-3,1) edge (-3,0);
 \node[label={$=$}] at (-1.75,0) {};

\node [fill,circle,draw,inner sep = 0pt, outer sep = 0pt, minimum size=2mm] (3) at (0,1) {};
\draw (0,-.5) edge node {} (3);
\draw (1,2) edge node {} (3);
\draw (-1,2) edge node {} (3);
\draw[fill=gRed] (1,2) circle [radius=0.5] node{$Y$};

\begin{scope}[xshift = 120, yshift = 20, xscale = 1.2]
\draw [black] (0,0) rectangle (0.5,0.5);
\draw [black,fill=gRed] (1.25,0) rectangle (9,0.5);

\node [fill,circle,draw,inner sep = 0pt, outer sep = 0pt, minimum size=1.5mm] (from) at (.25,.25) {};

\node [fill,circle,draw,inner sep = 0pt, outer sep = 0pt, minimum size=1.5mm] (to1) at (8.75,.25) {};
\node [fill,circle,draw,inner sep = 0pt, outer sep = 0pt, minimum size=1.5mm] (to2) at (8.5,.25) {};

\node [fill,circle,draw,inner sep = 0pt, outer sep = 0pt, minimum size=1.5mm] (to3) at (1.5,.25) {};
\node [fill,circle,draw,inner sep = 0pt, outer sep = 0pt, minimum size=1.5mm] (to4) at (1.75,.25) {};

\draw[->,bend left=60,thick] (from) to (to3);
\draw[->,bend left=60,thick] (from) to (to4);
\draw[->,bend left=50,thick] (from) to (to1);
\draw[->,bend left=50,thick] (from) to (to2);

\node at (5.125,0.25) {$\cdots$};
\end{scope}

 \node[label={$\Leftrightarrow$}] at (2.5,0.5) {};

\end{tikzpicture}
\end{figure}

and
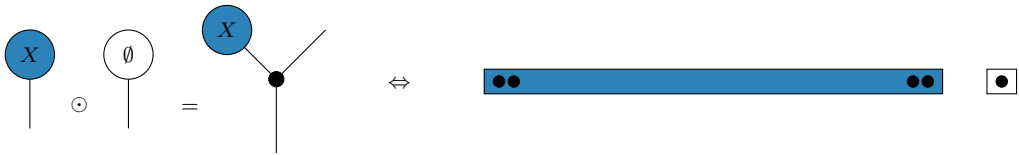
\begin{figure}[h!]
\centering
\begin{tikzpicture}[scale=0.65]
\draw[fill=gGreen] (-5,1.5) circle [radius=0.5] node{$X$};
 \node[label={$\odot$}] at (-4,0) {};
\draw[fill=white] (-3,1.5) circle [radius=0.5] node{$\emptyset$};
\draw (-5,1) edge (-5,0);
\draw (-3,1) edge (-3,0);
 \node[label={$=$}] at (-1.75,0) {};

\node [fill,circle,draw,inner sep = 0pt, outer sep = 0pt, minimum size=2mm] (3) at (0,1) {};
\draw (0,-.5) edge node {} (3);
\draw (1,2) edge node {} (3);
\draw (-1,2) edge node {} (3);
\draw[fill=gGreen] (-1,2) circle [radius=0.5] node{$X$};

\begin{scope}[xshift = 120, yshift = 20, xscale = 1.2]
\draw [black,fill=gGreen] (0,0) rectangle (7.75,0.5);
\draw [black] (8.5,0) rectangle (9,0.5);

\node [fill,circle,draw,inner sep = 0pt, outer sep = 0pt, minimum size=1.5mm] (to1) at (8.75,.25) {};
\node [fill,circle,draw,inner sep = 0pt, outer sep = 0pt, minimum size=1.5mm] (to3) at (7.25,.25) {};
\node [fill,circle,draw,inner sep = 0pt, outer sep = 0pt, minimum size=1.5mm] (to4) at (7.5,.25) {};
\node [fill,circle,draw,inner sep = 0pt, outer sep = 0pt, minimum size=1.5mm] at (0.25,.25) {};
\node [fill,circle,draw,inner sep = 0pt, outer sep = 0pt, minimum size=1.5mm] at (0.5,.25) {};

\end{scope}

 \node[label={$\Leftrightarrow$}] at (2.5,0.5) {};

\end{tikzpicture}
\end{figure}

\end{proof}

\newpage

\begin{example}\label{ex6}
One may use the above algorithm to compute the binary trees associated to the objects of $N_\infty(C_{p^2})$ as follows.

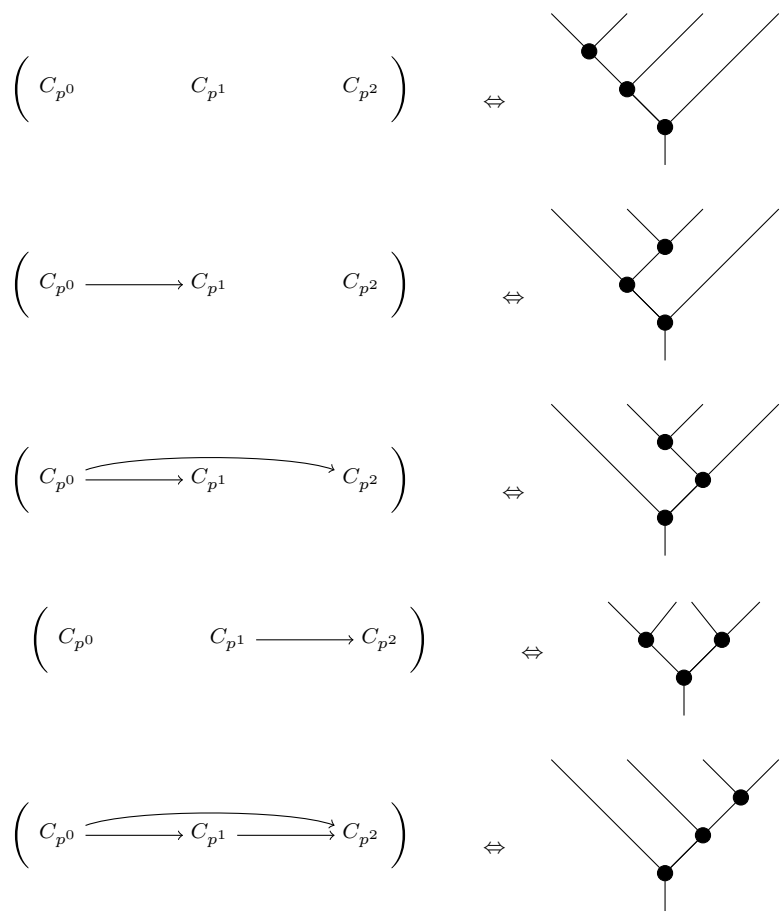
\begin{figure}[h!]
\centering
 \begin{tikzpicture}[->, node distance=2cm, auto]
\node (3) at (6.000000,0) {$C_{p^2}$};
\node (2) at (4.000000,0) {$C_{p^1}$};
\node (1) at (2.000000,0) {$C_{p^0}$};
 \node[label={$\Leftrightarrow$}] at (7.75,-0.45) {};
  \node at (1.5,0) {\Huge{(}};
\node at (6.5,0) {\Huge{)}};

 \begin{scope}[xscale=-1, xshift = -20cm]
 \node [fill,circle,draw,inner sep = 0pt, outer sep = 0pt, minimum size=2mm] (11) at (10.5,0) {};
\node [fill,circle,draw,inner sep = 0pt, outer sep = 0pt, minimum size=2mm] (33) at (10,-0.5) {};
\node [fill,circle,draw,inner sep = 0pt, outer sep = 0pt, minimum size=2mm] (55) at (11,0.5) {};
\draw[-] (10,-1) edge node {} (33);
\draw[-] (33) edge node {} (55);
\draw[-] (33) edge node {} (8.5,1);
\draw[-] (11) edge node {} (9.5,1);
\draw[-] (11) edge node {} (33);
\draw[-] (55) edge node {} (10.5,1);
\draw[-] (55) edge node {} (11.5,1);
 \end{scope}
 \end{tikzpicture}
\end{figure}

\begin{figure}[h!]
 \begin{tikzpicture}[->, node distance=2cm, auto]
\node (3) at (6.000000,0) {$C_{p^2}$};
\node (2) at (4.000000,0) {$C_{p^1}$};
\node (1) at (2.000000,0) {$C_{p^0}$};
\draw (1) to (2);
 \node[label={$\Leftrightarrow$}] at (8,-0.45) {};
   \node at (1.5,0) {\Huge{(}};
\node at (6.5,0) {\Huge{)}};

 \begin{scope}[xscale=-1, xshift = -20cm]
  \node [fill,circle,draw,inner sep = 0pt, outer sep = 0pt, minimum size=2mm] (11) at (10.5,0) {};
\node [fill,circle,draw,inner sep = 0pt, outer sep = 0pt, minimum size=2mm] (33) at (10,-0.5) {};
\node [fill,circle,draw,inner sep = 0pt, outer sep = 0pt, minimum size=2mm] (55) at (10,0.5) {};
\draw[-] (10,-1) edge node {} (33);
\draw[-] (33) edge node {} (8.5,1);
\draw[-] (11) edge node {} (9.5,1);
\draw[-] (11) edge node {} (33);
\draw[-] (55) edge node {} (10.5,1);
\draw[-] (33) edge node {} (11.5,1);
 \end{scope}
 \end{tikzpicture}
\end{figure}

\begin{figure}[h!]
 \begin{tikzpicture}[->, node distance=2cm, auto]
\node (3) at (6.000000,0) {$C_{p^2}$};
\node (2) at (4.000000,0) {$C_{p^1}$};
\node (1) at (2.000000,0) {$C_{p^0}$};
\draw (1) to (2);
\draw (1) edge[bend left=20, looseness=0.5, ->] (3);
 \node[label={$\Leftrightarrow$}] at (8,-0.45) {};
   \node at (1.5,0) {\Huge{(}};
\node at (6.5,0) {\Huge{)}};

  \node [fill,circle,draw,inner sep = 0pt, outer sep = 0pt, minimum size=2mm] (11) at (10.5,0) {};
\node [fill,circle,draw,inner sep = 0pt, outer sep = 0pt, minimum size=2mm] (33) at (10,-0.5) {};
\node [fill,circle,draw,inner sep = 0pt, outer sep = 0pt, minimum size=2mm] (55) at (10,0.5) {};
\draw[-] (10,-1) edge node {} (33);
\draw[-] (33) edge node {} (8.5,1);
\draw[-] (11) edge node {} (9.5,1);
\draw[-] (11) edge node {} (33);
\draw[-] (55) edge node {} (10.5,1);
\draw[-] (33) edge node {} (11.5,1);
 \end{tikzpicture}
  \end{figure}

\begin{figure}[h!]
  \begin{tikzpicture}[->, node distance=2cm, auto]
\node (3) at (6.000000,0) {$C_{p^2}$};
\node (2) at (4.000000,0) {$C_{p^1}$};
\node (1) at (2.000000,0) {$C_{p^0}$};
\draw (2) to (3);
 \node[label={$\Leftrightarrow$}] at (8,-0.45) {};
   \node at (1.5,0) {\Huge{(}};
\node at (6.5,0) {\Huge{)}};

   \node [fill,circle,draw,inner sep = 0pt, outer sep = 0pt, minimum size=2mm] (11) at (10.5,0) {};
\node [fill,circle,draw,inner sep = 0pt, outer sep = 0pt, minimum size=2mm] (33) at (10,-0.5) {};
\node [fill,circle,draw,inner sep = 0pt, outer sep = 0pt, minimum size=2mm] (55) at (9.5,0) {};
\draw[-] (10,-1) edge node {} (33);
\draw[-] (33) edge node {} (9,0.5);
\draw[-] (11) edge node {} (10.1,0.5);
\draw[-] (11) edge node {} (33);
\draw[-] (55) edge node {} (9.9,0.5);
\draw[-] (33) edge node {} (11,0.5);
 \end{tikzpicture}
 \end{figure}

\begin{figure}[h!]
  \begin{tikzpicture}[->, node distance=2cm, auto]
\node (3) at (6.000000,0) {$C_{p^2}$};
\node (2) at (4.000000,0) {$C_{p^1}$};
\node (1) at (2.000000,0) {$C_{p^0}$};
\draw (1) to (2);
\draw (2) to (3);
\draw (1) edge[bend left=20, looseness=0.5, ->] (3);
\node[label={$\Leftrightarrow$}] at (7.75,-0.45) {};
  \node at (1.5,0) {\Huge{(}};
\node at (6.5,0) {\Huge{)}};

 \node [fill,circle,draw,inner sep = 0pt, outer sep = 0pt, minimum size=2mm] (11) at (10.5,0) {};
\node [fill,circle,draw,inner sep = 0pt, outer sep = 0pt, minimum size=2mm] (33) at (10,-0.5) {};
\node [fill,circle,draw,inner sep = 0pt, outer sep = 0pt, minimum size=2mm] (55) at (11,0.5) {};
\draw[-] (10,-1) edge node {} (33);
\draw[-] (33) edge node {} (55);
\draw[-] (33) edge node {} (8.5,1);
\draw[-] (11) edge node {} (9.5,1);
\draw[-] (11) edge node {} (33);
\draw[-] (55) edge node {} (10.5,1);
\draw[-] (55) edge node {} (11.5,1);
 \end{tikzpicture}
 \end{figure}
\end{example}

\normalfont

\subsection{The relation to the associahedron}

We shall now see that the relationship between $N_\infty(C_{p^n})$ and the Catalan numbers runs deeper than just the result of Theorem~\ref{mainthm}. Recall that we can put an order on binary trees. Indeed, let $X$ and $Y$ be binary trees with $n+1$ edges. Then we say that $X < Y$ if $Y$ can be obtained from $X$ by a (finite sequence of) \emph{clockwise} tree rotation operations, i.e., by moving a branch from left to right.

\begin{figure}[h!]
\centering
\begin{tikzpicture}[scale=0.5]
\node [fill,circle,draw,inner sep = 0pt, outer sep = 0pt, minimum size=2mm] (1) at (5,1) {};
\node [fill,circle,draw,inner sep = 0pt, outer sep = 0pt, minimum size=2mm] (3) at (6,0) {};
\draw (6,-1) edge node {} (3);
\draw (3) edge node {} (8,2);
\draw (1) edge node {} (4,2);
\draw (1) edge node {} (6,2);
\draw (1) edge node {} (3);

 \node[label={$<$}] at (8.5,0) {};

\node [fill,circle,draw,inner sep = 0pt, outer sep = 0pt, minimum size=2mm] (11) at (12,1) {};
\node [fill,circle,draw,inner sep = 0pt, outer sep = 0pt, minimum size=2mm] (33) at (11,0) {};
\draw (11,-1) edge node {} (33);
\draw (33) edge node {} (13,2);
\draw (33) edge node {} (9,2);
\draw (11) edge node {} (11,2);
\draw (11) edge node {} (33);
\end{tikzpicture}
\caption{An example of an order relation between two binary trees.\label{fig:binaryorder}}
\end{figure}

A more general example is given by Figure \ref{fig:generaltreeorder}. 
Furthermore, if a binary tree $X$ were to contain the left hand side as a subtree, then
we could make a binary tree $Y$ by replacing that subtree with the right hand side. We would see that $Y$ is obtained from 
$X$ by a (finite sequence of) clockwise tree rotation operations, so $X<Y$.

\begin{figure}[h!]
\centering
\begin{tikzpicture}[scale=0.7]
\node [fill,circle,draw,inner sep = 0pt, outer sep = 0pt, minimum size=2mm] (1) at (4,1) {};
\node [fill,circle,draw,inner sep = 0pt, outer sep = 0pt, minimum size=2mm] (3) at (5,0) {};
\draw (5,-1) edge node {} (3);
\draw (3) edge node {} (7,2);
\draw (1) edge node {} (3,2);
\draw (1) edge node {} (5,2);
\draw (1) edge node {} (3);
\draw[fill=gRed] (7,2) circle [radius=0.5] node{$C$};
\draw[fill=gGreen] (5,2) circle [radius=0.5] node{$B$};
\draw[fill=gYellow] (3,2) circle [radius=0.5] node{$A$};
 \node[label={$<$}] at (9,0) {};
\node [fill,circle,draw,inner sep = 0pt, outer sep = 0pt, minimum size=2mm] (11) at (14,1) {};
\node [fill,circle,draw,inner sep = 0pt, outer sep = 0pt, minimum size=2mm] (33) at (13,0) {};
\draw (13,-1) edge node {} (33);
\draw (33) edge node {} (15,2);
\draw (33) edge node {} (11,2);
\draw (11) edge node {} (13,2);
\draw (11) edge node {} (33);
\draw[fill=gYellow] (11,2) circle [radius=0.5] node{$A$};
\draw[fill=gGreen] (13,2) circle [radius=0.5] node{$B$};
\draw[fill=gRed] (15,2) circle [radius=0.5] node{$C$};
\end{tikzpicture}
\caption{A general example of on order relation between two binary trees.\label{fig:generaltreeorder}}
\end{figure}

The poset structure on the set of binary trees with $n+1$ edges is known as the $n$-associaheadron, see Stasheff~\cite{MR0158400}. We shall denote this poset structure as $\mathcal{A}_{n}$.

\vskip2cm

We can also implement a poset structure on $N_\infty(C_{p^n})$ by fixing that $X < Y$ if $Y$ can be obtained from $X$ via the addition of norm maps. For example, we have the following.
\begin{figure}[h!]
 \begin{tikzpicture}[->, node distance=2cm, auto]
\node (1) at (6.000000,0) {$C_{p^1}$};
\node (2) at (4.000000,0) {$C_{p^0}$};
  \node at (3.5,0) {\Huge{(}};
\node at (6.5,0) {\Huge{)}};

 \node[label={$<$}] at (7,-0.4) {};

\node (11) at (10.000000,0) {$C_{p^1}$};
\node (22) at (8.000000,0) {$C_{p^0}$};
  \node at (7.5,0) {\Huge{(}};
\node at (10.5,0) {\Huge{)}};
\draw (22) to (11);
 \end{tikzpicture}
\end{figure}


\begin{remark}\label{rem:order}
Note that $Z \odot -$ and $- \odot Z$ preserve this ordering. That is, if $X < Y$ then it follows that $Z \odot X < Z \odot Y$ and $X \odot Z < Y \odot Z$.
\end{remark}

We now prove our main theorem which tells us that these poset structures actually agree.

\begin{thm}
There is an order-preserving and order-reflecting bijection of posets
$$\{N_\infty (C_{p^n} )\} \Leftrightarrow \{\text{rooted binary trees with n+2 \text{leaves}} \} \Leftrightarrow \mathcal{A}_{n+1}.$$
\end{thm}

\begin{proof}
Let us begin by showing that a clockwise tree rotation corresponds to the addition of an arrow in the corresponding $N_\infty$-diagram, or more specifically, the addition of a norm map. We shall do this by appealing to the diagrammatic representations, as it provides the cleanest proof. 
Consider a branch move as in Figure \ref{fig:generaltreeorder}. 
The left hand tree corresponds to the first norm diagram below, which we may call $(A \odot B) \odot C$. 
The right hand tree corresponds to the second (where restrictions of the largest arrow
are omitted for clarity), which is $A \odot (B \odot C)$. 

\begin{figure}[h!]
\centering
\begin{tikzpicture}
\draw [black,fill=gGreen] (3,0) rectangle (5,0.5);
\draw [black] (5.75,0) rectangle (6.25,0.5);
\draw [black,fill=gRed] (7,0) rectangle (9,0.5);
\draw [black] (1.75,0) rectangle (2.25,0.5);
\draw [black,fill=gYellow] (-1,0) rectangle (1,0.5);

\node [fill,circle,draw,inner sep = 0pt, outer sep = 0pt, minimum size=1.5mm] at (0.75,.25) {};
\node [fill,circle,draw,inner sep = 0pt, outer sep = 0pt, minimum size=1.5mm] at (0.5,.25) {};
\node [fill,circle,draw,inner sep = 0pt, outer sep = 0pt, minimum size=1.5mm] at (-0.75,.25) {};
\node [fill,circle,draw,inner sep = 0pt, outer sep = 0pt, minimum size=1.5mm] at (-0.5,.25) {};
\node at (0.025,0.25) {$\cdots$};

\node [fill,circle,draw,inner sep = 0pt, outer sep = 0pt, minimum size=1.5mm] (Afrom) at (2,.25) {};
\node [fill,circle,draw,inner sep = 0pt, outer sep = 0pt, minimum size=1.5mm] (Ato1) at (4.75,.25) {};
\node [fill,circle,draw,inner sep = 0pt, outer sep = 0pt, minimum size=1.5mm] (Ato2) at (4.5,.25) {};
\node [fill,circle,draw,inner sep = 0pt, outer sep = 0pt, minimum size=1.5mm] (Ato3) at (3.25,.25) {};
\node [fill,circle,draw,inner sep = 0pt, outer sep = 0pt, minimum size=1.5mm] (Ato4) at (3.5,.25) {};

\draw[->,bend left=80,thick] (Afrom) to (Ato3);
\draw[->,bend left=80,thick] (Afrom) to (Ato4);
\draw[->,bend left=80,thick] (Afrom) to (Ato1);
\draw[->,bend left=80,thick] (Afrom) to (Ato2);
\node at (4.025,0.25) {$\cdots$};

\node [fill,circle,draw,inner sep = 0pt, outer sep = 0pt, minimum size=1.5mm] (from) at (6,.25) {};
\node [fill,circle,draw,inner sep = 0pt, outer sep = 0pt, minimum size=1.5mm] (to1) at (8.75,.25) {};
\node [fill,circle,draw,inner sep = 0pt, outer sep = 0pt, minimum size=1.5mm] (to2) at (8.5,.25) {};
\node [fill,circle,draw,inner sep = 0pt, outer sep = 0pt, minimum size=1.5mm] (to3) at (7.25,.25) {};
\node [fill,circle,draw,inner sep = 0pt, outer sep = 0pt, minimum size=1.5mm] (to4) at (7.5,.25) {};

\draw[->,bend left=80,thick] (from) to (to3);
\draw[->,bend left=80,thick] (from) to (to4);
\draw[->,bend left=80,thick] (from) to (to1);
\draw[->,bend left=80,thick] (from) to (to2);
\node at (8.025,0.25) {$\cdots$};

\end{tikzpicture}
\begin{tikzpicture}
\draw [black,fill=gGreen] (3,0) rectangle (5,0.5);
\draw [black] (5.75,0) rectangle (6.25,0.5);
\draw [black,fill=gRed] (7,0) rectangle (9,0.5);
\draw [black] (1.75,0) rectangle (2.25,0.5);
\draw [black,fill=gYellow] (-1,0) rectangle (1,0.5);

\node [fill,circle,draw,inner sep = 0pt, outer sep = 0pt, minimum size=1.5mm] at (0.75,.25) {};
\node [fill,circle,draw,inner sep = 0pt, outer sep = 0pt, minimum size=1.5mm] at (0.5,.25) {};
\node [fill,circle,draw,inner sep = 0pt, outer sep = 0pt, minimum size=1.5mm] at (-0.75,.25) {};
\node [fill,circle,draw,inner sep = 0pt, outer sep = 0pt, minimum size=1.5mm] at (-0.5,.25) {};
\node at (0.025,0.25) {$\cdots$};

\node [fill,circle,draw,inner sep = 0pt, outer sep = 0pt, minimum size=1.5mm] (Afrom) at (2,.25) {};
\node [fill,circle,draw,inner sep = 0pt, outer sep = 0pt, minimum size=1.5mm] (Ato1) at (4.75,.25) {};
\node [fill,circle,draw,inner sep = 0pt, outer sep = 0pt, minimum size=1.5mm] (Ato2) at (4.5,.25) {};
\node [fill,circle,draw,inner sep = 0pt, outer sep = 0pt, minimum size=1.5mm] (Ato3) at (3.25,.25) {};
\node [fill,circle,draw,inner sep = 0pt, outer sep = 0pt, minimum size=1.5mm] (Ato4) at (3.5,.25) {};

\draw[->,bend left=80,thick] (Afrom) to (Ato3);
\draw[->,bend left=80,thick] (Afrom) to (Ato4);
\draw[->,bend left=80,thick] (Afrom) to (Ato1);
\draw[->,bend left=80,thick] (Afrom) to (Ato2);
\node at (4.025,0.25) {$\cdots$};

\node [fill,circle,draw,inner sep = 0pt, outer sep = 0pt, minimum size=1.5mm] (from) at (6,.25) {};
\node [fill,circle,draw,inner sep = 0pt, outer sep = 0pt, minimum size=1.5mm] (to1) at (8.75,.25) {};
\node [fill,circle,draw,inner sep = 0pt, outer sep = 0pt, minimum size=1.5mm] (to2) at (8.5,.25) {};
\node [fill,circle,draw,inner sep = 0pt, outer sep = 0pt, minimum size=1.5mm] (to3) at (7.25,.25) {};
\node [fill,circle,draw,inner sep = 0pt, outer sep = 0pt, minimum size=1.5mm] (to4) at (7.5,.25) {};

\draw[->,bend left=80,thick] (from) to (to3);
\draw[->,bend left=80,thick] (from) to (to4);
\draw[->,bend left=80,thick] (from) to (to1);
\draw[->,bend left=80,thick] (from) to (to2);
\node at (8.025,0.25) {$\cdots$};

\draw[->,bend left=80,thick] (Afrom) to (to1);
\end{tikzpicture}
\caption{The diagrams corresponding to the trees in Figure \ref{fig:generaltreeorder}.}\label{fig:after9}
\end{figure}

\vspace{3mm}

We now show that adding an arrow in a norm diagram induces a clockwise tree rotation in the corresponding binary trees. We shall do this by induction on $n$. Note that the base case can be easily checked, see Figure~\ref{fig:binaryorder} and the corresponding discussion below it. Example \ref{ex6} then illustrates the next case. Suppose that we begin with an arbitrary norm diagram $A \odot Z$ to which we have added the red arrow. 

\begin{figure}[h!]
\centering
\begin{tikzpicture}[scale = 0.8]
\draw [black,fill=gGreen] (0,0) rectangle (5,0.5);
\draw [black] (5.75,0) rectangle (6.25,0.5);
\draw [black,fill=gRed] (7,0) rectangle (9,0.5);

\node [fill,circle,draw,inner sep = 0pt, outer sep = 0pt, minimum size=1.5mm] (from) at (6,.25) {};

\node [fill,circle,draw,inner sep = 0pt, outer sep = 0pt, minimum size=1.5mm] (to1) at (8.75,.25) {};
\node [fill,circle,draw,inner sep = 0pt, outer sep = 0pt, minimum size=1.5mm] (to2) at (8.5,.25) {};

\node [fill,circle,draw,inner sep = 0pt, outer sep = 0pt, minimum size=1.5mm] (to3) at (7.25,.25) {};
\node [fill,circle,draw,inner sep = 0pt, outer sep = 0pt, minimum size=1.5mm] (to4) at (7.5,.25) {};

\node [fill,circle,draw,inner sep = 0pt, outer sep = 0pt, minimum size=1.5mm] (from1) at (4.75,.25) {};
\node [fill,circle,draw,inner sep = 0pt, outer sep = 0pt, minimum size=1.5mm] at (4.5,.25) {};

\node [fill,circle,draw,inner sep = 0pt, outer sep = 0pt, minimum size=1.5mm] at (0.25,.25) {};
\node [fill,circle,draw,inner sep = 0pt, outer sep = 0pt, minimum size=1.5mm] at (0.5,.25) {};

\node at (8.025,0.25) (to21) {$\cdots$};
\node at (2.525,0.25) (from12) {$\cdots$};

\draw[->,bend left=80,thick] (from) to (to1);
\draw[->, bend right=40,thick, gRed2] (from12) to (to21);
\end{tikzpicture}
\end{figure}

We can assume that the new arrow starts in $A$ and ends in $Z$. 
By the composition and restriction rule, we can without loss of generality assume that the new arrow
ends at the pivot (the added vertex in $A \odot Z$). 

\begin{figure}[h!]
\centering
\begin{tikzpicture}[scale = 0.8]
\draw [black,fill=gGreen] (0,0) rectangle (5,0.5);
\draw [black] (5.75,0) rectangle (6.25,0.5);
\draw [black,fill=gRed] (7,0) rectangle (9,0.5);

\node [fill,circle,draw,inner sep = 0pt, outer sep = 0pt, minimum size=1.5mm] (from) at (6,.25) {};

\node [fill,circle,draw,inner sep = 0pt, outer sep = 0pt, minimum size=1.5mm] (to1) at (8.75,.25) {};
\node [fill,circle,draw,inner sep = 0pt, outer sep = 0pt, minimum size=1.5mm] (to2) at (8.5,.25) {};

\node [fill,circle,draw,inner sep = 0pt, outer sep = 0pt, minimum size=1.5mm] (to3) at (7.25,.25) {};
\node [fill,circle,draw,inner sep = 0pt, outer sep = 0pt, minimum size=1.5mm] (to4) at (7.5,.25) {};

\node [fill,circle,draw,inner sep = 0pt, outer sep = 0pt, minimum size=1.5mm] (from1) at (4.75,.25) {};
\node [fill,circle,draw,inner sep = 0pt, outer sep = 0pt, minimum size=1.5mm] at (4.5,.25) {};

\node [fill,circle,draw,inner sep = 0pt, outer sep = 0pt, minimum size=1.5mm] at (0.25,.25) {};
\node [fill,circle,draw,inner sep = 0pt, outer sep = 0pt, minimum size=1.5mm] at (0.5,.25) {};

\node at (8.025,0.25) {$\cdots$};
\node at (2.525,0.25) (from12) {$\cdots$};

\draw[->,bend left=80,thick] (from) to (to1);
\draw[->, bend right=40,thick, gRed2] (from12) to (from);
\end{tikzpicture}
\end{figure}

We can split up the left hand block into a diagram of the form $X \odot Y$ for some smaller diagrams $X$ and $Y$. 
This gives three different cases to consider based on where the new arrow begins.
These situations are summarised in Figure~\ref{casesthree}. In particular, the new arrow could
start in $Y$, giving Case 1, it could start in $X$, giving Case 3, or the final option is that the new arrow begins at the vertex arising from the $\odot$ operation in $X \odot Y$.

\begin{figure}[h!]
\centering
\begin{tikzpicture}
\draw [black,fill=gGreen] (3,0) rectangle (5,0.5);
\draw [black] (5.75,0) rectangle (6.25,0.5);
\draw [black,fill=gRed] (7,0) rectangle (9,0.5);
\draw [black] (1.75,0) rectangle (2.25,0.5);
\draw [black,fill=gYellow] (-1,0) rectangle (1,0.5);

\node [fill,circle,draw,inner sep = 0pt, outer sep = 0pt, minimum size=1.5mm] at (0.75,.25) {};
\node [fill,circle,draw,inner sep = 0pt, outer sep = 0pt, minimum size=1.5mm] at (0.5,.25) {};
\node [fill,circle,draw,inner sep = 0pt, outer sep = 0pt, minimum size=1.5mm]  at (-0.75,.25) {};
\node [fill,circle,draw,inner sep = 0pt, outer sep = 0pt, minimum size=1.5mm] at (-0.5,.25) {};
\node at (0.025,0.25) (Bfrom) {$\cdots$};

\node [fill,circle,draw,inner sep = 0pt, outer sep = 0pt, minimum size=1.5mm] (Afrom) at (2,.25) {};
\node [fill,circle,draw,inner sep = 0pt, outer sep = 0pt, minimum size=1.5mm] (Ato1) at (4.75,.25) {};
\node [fill,circle,draw,inner sep = 0pt, outer sep = 0pt, minimum size=1.5mm] (Ato2) at (4.5,.25) {};
\node [fill,circle,draw,inner sep = 0pt, outer sep = 0pt, minimum size=1.5mm] (Ato3) at (3.25,.25) {};
\node [fill,circle,draw,inner sep = 0pt, outer sep = 0pt, minimum size=1.5mm] (Ato4) at (3.5,.25) {};

\draw[->,bend left=80,thick] (Afrom) to (Ato1);
\node at (4.025,0.25) (Ato12) {$\cdots$};

\node [fill,circle,draw,inner sep = 0pt, outer sep = 0pt, minimum size=1.5mm] (from) at (6,.25) {};
\node [fill,circle,draw,inner sep = 0pt, outer sep = 0pt, minimum size=1.5mm] (to1) at (8.75,.25) {};
\node [fill,circle,draw,inner sep = 0pt, outer sep = 0pt, minimum size=1.5mm] (to2) at (8.5,.25) {};
\node [fill,circle,draw,inner sep = 0pt, outer sep = 0pt, minimum size=1.5mm] (to3) at (7.25,.25) {};
\node [fill,circle,draw,inner sep = 0pt, outer sep = 0pt, minimum size=1.5mm] (to4) at (7.5,.25) {};

\draw[->,bend left=80,thick] (from) to (to1);
\node at (8.025,0.25) {$\cdots$};

\draw[->, bend right=40,thick, gRed2] (Ato12)  to node[above,black] {1} (from);
\draw[->, bend right=40,thick, gRed2] (Afrom) to node[above,black] {2} (from);
\draw[->, bend right=40,thick, gRed2] (Bfrom) to node[above,black] {3} (from);
\end{tikzpicture}
\caption{The three cases for adding a non-trivial norm map to $(X \odot Y) \odot Z$}\label{casesthree}
\end{figure}


Case 2 has already been covered in the beginning of this proof, 
as adding an arrow connecting the two pivots is equivalent to adding an arrow from the first pivot
to the right most vertex. Therefore, it is the situation described in Figure \ref{fig:generaltreeorder} and Figure \ref{fig:after9}. 
In that part we illustrated the corresponding tree move 
to the addition of such an arrow.


For Case 1, 
it is important to note that adding the arrow (1) creates the arrow (2) via composition, 
but that adding the arrow (2) does \emph{not} create arrow (1) by any of the rules. 
However, we can say that adding arrow (1) is equivalent to first adding arrow (2) and then arrow (1).

Hence, we add arrow (2) to $(X \odot Y) \odot Z$. 
By Case 2, this gives $X \odot (Y \odot Z)$
and in terms of trees gives Figure \ref{fig:generaltreeorder}.
We then add arrow (1). This only affects the term $Y \odot Z$, giving a new diagram $T$
with $Y \odot Z < T$. Moreover, $Y \odot Z$ and $T$ are in 
$\{N_\infty (C_{p^k} )\}$ for $k < n$. 
By induction, we have that the trees corresponding to 
$Y \odot Z$ and $T$ are ordered correctly.  

\begin{figure}[h!]
\centering
\begin{tikzpicture}[scale=0.7]
\node [fill,circle,draw,inner sep = 0pt, outer sep = 0pt, minimum size=2mm] (3) at (0,1) {};
\draw (0,-.5) edge node {} (3);
\draw (1,2) edge node {} (3);
\draw (-1,2) edge node {} (3);
\draw[fill=gGreen] (-1,2) circle [radius=0.5] node{$Y$};
\draw[fill=gRed] (1,2) circle [radius=0.5] node{$Z$};

 \node[label={$<$}] at (3,1) {};

\node [fill,circle,draw,inner sep = 0pt, outer sep = 0pt, minimum size=2mm] (7) at (5,2) {};
\draw (5,-.5) edge node {} (7);
\draw[fill=Yellow] (5,2) circle [radius=0.5] node{$T$};

\end{tikzpicture}
\end{figure}


As the order on norm maps is preserved by $X \odot -$, 
we have $X \odot(Y \odot Z) < X \odot T$, see Remark \ref{rem:order}.
The corresponding operation on trees also preserves the order to give
the following.

\begin{figure}[h!]
\centering
\begin{tikzpicture}[scale=0.7]
\node [fill,circle,draw,inner sep = 0pt, outer sep = 0pt, minimum size=2mm] (1) at (4,1) {};
\node [fill,circle,draw,inner sep = 0pt, outer sep = 0pt, minimum size=2mm] (3) at (5,0) {};
\draw (5,-1) edge node {} (3);
\draw (3) edge node {} (7,2);
\draw (1) edge node {} (3,2);
\draw (1) edge node {} (5,2);
\draw (1) edge node {} (3);
\draw[fill=gRed] (7,2) circle [radius=0.5] node{$Z$};
\draw[fill=gGreen] (5,2) circle [radius=0.5] node{$Y$};
\draw[fill=gYellow] (3,2) circle [radius=0.5] node{$X$};
 \node[label={$<$}] at (8,1) {};

\node [fill,circle,draw,inner sep = 0pt, outer sep = 0pt, minimum size=2mm] (11) at (12,1) {};
\node [fill,circle,draw,inner sep = 0pt, outer sep = 0pt, minimum size=2mm] (33) at (11,0) {};
\draw (11,-1) edge node {} (33);
\draw (33) edge node {} (13,2);
\draw (33) edge node {} (9,2);
\draw (11) edge node {} (11,2);
\draw (11) edge node {} (33);
\draw[fill=gYellow] (9,2) circle [radius=0.5] node{$X$};
\draw[fill=gGreen] (11,2) circle [radius=0.5] node{$Y$};
\draw[fill=gRed] (13,2) circle [radius=0.5] node{$Z$};

 \node[label={$<$}] at (14,1) {};
\node [fill,circle,draw,inner sep = 0pt, outer sep = 0pt, minimum size=2mm] (101) at (17,0) {};
\draw (17,-1) edge node {} (101);
\draw (101) edge node {} (15,2);
\draw (101) edge node {} (19,2);
\draw[fill=gYellow] (15,2) circle [radius=0.5] node{$X$};
\draw[fill=Yellow] (19,2) circle [radius=0.5] node{$T$};
\end{tikzpicture}
\end{figure}



For Case 3, we repeat our earlier decomposition. Either, the new arrow starts at the leftmost vertex
or we can split $X$ into $X_1 \odot X_2$ and see that we are in Case 1 or Case 2 (which are solved)
or we are in Case 3 for this new decomposition. For Case 3, we can continue to split up $X_1$ into smaller diagrams $X_1 = X_1' \odot X_2'$, and so on. Continuing in this way, the only new case is 
the following, which we recognise as adding the arrow ($3'$) from the leftmost vertex to 
the pivot of $Y \odot Z$.

\begin{figure}[h!]
\centering
\begin{tikzpicture}
\draw [black,fill=gGreen] (3,0) rectangle (5,0.5);
\draw [black] (5.75,0) rectangle (6.25,0.5);
\draw [black,fill=gRed] (7,0) rectangle (9,0.5);
\draw [black] (1.75,0) rectangle (2.25,0.5);
\draw [black,fill=gYellow] (-1,0) rectangle (1,0.5);

\node [fill,circle,draw,inner sep = 0pt, outer sep = 0pt, minimum size=1.5mm] (Xtarget) at (0.75,.25) {};
\node [fill,circle,draw,inner sep = 0pt, outer sep = 0pt, minimum size=1.5mm] at (0.5,.25) {};
\node [fill,circle,draw,inner sep = 0pt, outer sep = 0pt, minimum size=1.5mm]  at (-0.75,.25) {};
\node [fill,circle,draw,inner sep = 0pt, outer sep = 0pt, minimum size=1.5mm]  at (-0.5,.25) {};
\node at (0.025,0.25) (Bfrom) {$\cdots$};

\node [fill,circle,draw,inner sep = 0pt, outer sep = 0pt, minimum size=1.5mm] (Afrom) at (2,.25) {};
\node [fill,circle,draw,inner sep = 0pt, outer sep = 0pt, minimum size=1.5mm] (Ato1) at (4.75,.25) {};
\node [fill,circle,draw,inner sep = 0pt, outer sep = 0pt, minimum size=1.5mm] (Ato2) at (4.5,.25) {};
\node [fill,circle,draw,inner sep = 0pt, outer sep = 0pt, minimum size=1.5mm] (Ato3) at (3.25,.25) {};
\node [fill,circle,draw,inner sep = 0pt, outer sep = 0pt, minimum size=1.5mm] (Ato4) at (3.5,.25) {};

\draw[->,bend left=80,thick] (Afrom) to (Ato1);
\node at (4.025,0.25) (Ato12) {$\cdots$};

\node [fill,circle,draw,inner sep = 0pt, outer sep = 0pt, minimum size=1.5mm] (from) at (6,.25) {};
\node [fill,circle,draw,inner sep = 0pt, outer sep = 0pt, minimum size=1.5mm] (to1) at (8.75,.25) {};
\node [fill,circle,draw,inner sep = 0pt, outer sep = 0pt, minimum size=1.5mm] (to2) at (8.5,.25) {};
\node [fill,circle,draw,inner sep = 0pt, outer sep = 0pt, minimum size=1.5mm] (to3) at (7.25,.25) {};
\node [fill,circle,draw,inner sep = 0pt, outer sep = 0pt, minimum size=1.5mm] (to4) at (7.5,.25) {};

\draw[->,bend left=80,thick] (from) to (to1);
\node at (8.025,0.25) {$\cdots$};

\draw[->, bend right=60,thick, gRed2] (-0.75,.25) to node[above,black] {$3'$} (from);

\end{tikzpicture}
\end{figure}

%

The arrow ($3'$) adds the arrow ($\alpha$) in the diagram below by restriction (but not vice versa). Therefore, we can obtain our diagram by first adding ($\alpha$) and then ($3'$). The diagram $X \odot Y$ with ($\alpha$) added is of the form $\emptyset \odot B$ for some other $B$. In particular, $ X \odot Y \leq \emptyset \odot B$, which by induction induces the correct ordering on the corresponding trees. Using the order preserving properties of $\odot$ we have $(X \odot Y) \odot Z \leq (\emptyset \odot B) \odot Z$. 

If we now add the arrow ($3'$) to $((\emptyset \odot B) \odot Z$, we are adding a new arrow from a pivot to a pivot, which is Case 2. 

\begin{figure}[h!]
\centering
\begin{tikzpicture}

\draw [black,fill=gGreen] (3,0) rectangle (5,0.5);
\draw [black] (5.75,0) rectangle (6.25,0.5);
\draw [black,fill=gRed] (7,0) rectangle (9,0.5);
\draw [black] (1.75,0) rectangle (2.25,0.5);
\draw [black,fill=Yellow] (-1,0) rectangle (1,0.5);
\draw [black] (-2.25,0) rectangle (-1.75,0.5);

\node [fill,circle,draw,inner sep = 0pt, outer sep = 0pt, minimum size=1.5mm] (Xtarget) at (0.75,.25) {};
\node [fill,circle,draw,inner sep = 0pt, outer sep = 0pt, minimum size=1.5mm] at (0.5,.25) {};
\node [fill,circle,draw,inner sep = 0pt, outer sep = 0pt, minimum size=1.5mm]  at (-0.75,.25) {};
\node [fill,circle,draw,inner sep = 0pt, outer sep = 0pt, minimum size=1.5mm]  at (-0.5,.25) {};
\node [fill,circle,draw,inner sep = 0pt, outer sep = 0pt, minimum size=1.5mm] (start) at (-2,.25) {};
\node at (0.025,0.25) (Bfrom) {$\cdots$};

\node [fill,circle,draw,inner sep = 0pt, outer sep = 0pt, minimum size=1.5mm] (Afrom) at (2,.25) {};
\node [fill,circle,draw,inner sep = 0pt, outer sep = 0pt, minimum size=1.5mm] (Ato1) at (4.75,.25) {};
\node [fill,circle,draw,inner sep = 0pt, outer sep = 0pt, minimum size=1.5mm] (Ato2) at (4.5,.25) {};
\node [fill,circle,draw,inner sep = 0pt, outer sep = 0pt, minimum size=1.5mm] (Ato3) at (3.25,.25) {};
\node [fill,circle,draw,inner sep = 0pt, outer sep = 0pt, minimum size=1.5mm] (Ato4) at (3.5,.25) {};

\draw[->,bend left=80,thick] (Afrom) to (Ato1);
\node at (4.025,0.25) (Ato12) {$\cdots$};

\node [fill,circle,draw,inner sep = 0pt, outer sep = 0pt, minimum size=1.5mm] (from) at (6,.25) {};
\node [fill,circle,draw,inner sep = 0pt, outer sep = 0pt, minimum size=1.5mm] (to1) at (8.75,.25) {};
\node [fill,circle,draw,inner sep = 0pt, outer sep = 0pt, minimum size=1.5mm] (to2) at (8.5,.25) {};
\node [fill,circle,draw,inner sep = 0pt, outer sep = 0pt, minimum size=1.5mm] (to3) at (7.25,.25) {};
\node [fill,circle,draw,inner sep = 0pt, outer sep = 0pt, minimum size=1.5mm] (to4) at (7.5,.25) {};

\draw[->,bend left=80,thick] (from) to (to1);
\node at (8.025,0.25) {$\cdots$};

\draw[->, bend right=50,thick, gRed2] (start) to node[above,black] {$\alpha$} (Ato1);
\draw[->, bend right=60,thick, gRed2] (start) to node[above,black] {$3'$} (from);

\end{tikzpicture}
\caption{The arrow additions for Case 3.\label{fig:case3diagram}}
\end{figure}

\end{proof}

\normalfont

\section{Generalising to other cyclic groups}\label{sec:generalising}

We would like to have a closed formula for the cardinality of $N_\infty(G)$ for all finite cyclic $G$. We shall explore the obstructions to obtaining such a result in this section. The main result is the construction of a lower bound of the number of such operads for $G=C_{p^{n_1}_1} \cdots C_{p^{n_k}_k}$. Let us highlight the style of norm maps that we must deal with in this situation. Figure~\ref{cpq} gives the 10 possible $N_\infty$-operads for $G = C_{pq}$.

\begin{figure}[h!]
 \begin{tikzpicture}[->, node distance=2cm, auto,scale = 0.5]
\node (11) at (2.000000,0) {$C_{p}$};
\node (p1) at (0.000000,0) {$C_{1}$};
\node (q1) at (2.000000,-2.0) {$C_{pq}$};
\node (pq1) at (0.000000,-2.0) {$C_{q}$};

\node (12) at (7.000000,0) {$C_{p}$};
\node (p2) at (5.000000,0) {$C_{1}$};
\node (q2) at (7.000000,-2.0) {$C_{pq}$};
\node (pq2) at (5.000000,-2.0) {$C_{q}$};
\draw (p2) to (12);

\node (13) at (12.000000,0) {$C_{p}$};
\node (p3) at (10.000000,0) {$C_{1}$};
\node (q3) at (12.000000,-2.0) {$C_{pq}$};
\node (pq3) at (10.000000,-2.0) {$C_{q}$};
\draw (p3) to (pq3);

\node (14) at (17.000000,-0) {$C_{p}$};
\node (p4) at (15.000000,-0) {$C_{1}$};
\node (q4) at (17.000000,-2.0) {$C_{pq}$};
\node (pq4) at (15.000000,-2.0) {$C_{q}$};
\draw (pq4) to (q4);
\draw (p4) to (14);

\node (15) at (22.000000,-0) {$C_{p}$};
\node (p5) at (20.000000,-0) {$C_{1}$};
\node (q5) at (22.000000,-2) {$C_{pq}$};
\node (pq5) at (20.000000,-2) {$C_{q}$};
\draw (15) to (q5);
\draw (p5) to (pq5);

\node (16) at (2.000000,-5) {$C_{p}$};
\node (p6) at (0.000000,-5) {$C_{1}$};
\node (q6) at (2.000000,-7) {$C_{pq}$};
\node (pq6) at (0.000000,-7) {$C_{q}$};
\draw(p6) to (pq6);
\draw(p6) to (16);

\node (17) at (7.000000,-5) {$C_{p}$};
\node (p7) at (5.000000,-5) {$C_{1}$};
\node (q7) at (7.000000,-7.0) {$C_{pq}$};
\node (pq7) at (5.000000,-7.0) {$C_{q}$};
\draw (p7) to (q7);
\draw (p7) to (pq7);
\draw (pq7) to (q7);
\draw (p7) to (17);

\node (18) at (12.000000,-5) {$C_{p}$};
\node (p8) at (10.000000,-5) {$C_{1}$};
\node (q8) at (12.000000,-7.0) {$C_{pq}$};
\node (pq8) at (10.000000,-7.0) {$C_{q}$};
\draw (p8) to (18);
\draw (p8) to (q8);
\draw (18) to (q8);
\draw (p8) to (pq8);

\node (19) at (17.000000,-5) {$C_{p}$};
\node (p9) at (15.000000,-5) {$C_{1}$};
\node (q9) at (17.000000,-7) {$C_{pq}$};
\node (pq9) at (15.000000,-7) {$C_{q}$};
\draw (p9) to (19);
\draw (p9) to (q9);
\draw (p9) to (pq9);

\node (110) at (22.000000,-5) {$C_{p}$};
\node (p10) at (20.000000,-5) {$C_{1}$};
\node (q10) at (22.000000,-7) {$C_{pq}$};
\node (pq10) at (20.000000,-7) {$C_{q}$};
\draw (p10) to (110);
\draw (p10) to (q10);
\draw (p10) to (pq10);
\draw (pq10) to (q10);
\draw (110) to (q10);
\end{tikzpicture}
\caption{The 10 possible $N_\infty$-operad structures for $G=C_{pq}$.}\label{cpq}
\end{figure}

A key observation to make is that there is an ``odd one out'' among these diagrams. In particular, consider the following.

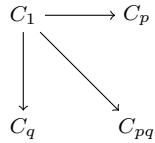
\vspace{5mm}

\begin{figure}[h!]
 \begin{tikzpicture}[->, node distance=2cm, auto,scale = 0.75]
\node (11) at (2.000000,0) {$C_{p}$};
\node (p1) at (0.000000,0) {$C_{1}$};
\node (q1) at (2.000000,-2.0) {$C_{pq}$};
\node (pq1) at (0.000000,-2.0) {$C_{q}$};
\draw (p1) to (11);
\draw (p1) to (q1);
\draw (p1) to (pq1);
\end{tikzpicture}
\end{figure}

\vspace{5mm}

This transfer system is different from the other nine because it is the only one where the diagonal is not forced by the composition and restriction rules of Corollary~\ref{operad-struc}. That is, if we were to remove the norm map $N^{pq}_1$, then the resulting diagram is still a valid $N_\infty$-operad. It follows that this $N_\infty$-operad cannot be formed by just combining those for $G=C_p$ and $G=C_q$. We will call such an operad \emph{mixed}. If it can be obtained from the component groups, then we will call it \emph{pure}.

The main result of this section will be to give a closed expression for the number of pure $N_\infty$-operads for $G=C_{p_1^{n_1} p_2^{n_2} \cdots p_k^{n_k}}$, which provides a non-trivial lower bound for the total number of $N_\infty$-operads for $G$.

Trying to manually enumerate the norm maps for $G = C_{p^n q^m}$, $p \neq q$ or even just $C_{p^3q}$ shows that the situation is already intangibly complicated. Indeed, we have computationally verified that there are 544 $N_\infty$-operads for $C_{p^3q}$.

\subsection{Enumerating pure operads}

We begin with a more formal definition of ``pure'' and ``mixed''.

Let $Z$ be an $N_\infty$-operad for $G = C_{p^nq^m}$. That it, $Z$ is an $N_\infty$-diagram on the lattice below.

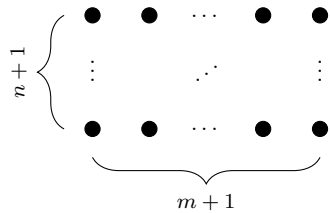
\vspace{5mm}

\begin{figure}[h!]
\begin{tikzpicture}[scale = 0.75]
\node [fill,circle,draw,inner sep = 0pt, outer sep = 0pt, minimum size=2mm] at (0,0) {};
\node [fill,circle,draw,inner sep = 0pt, outer sep = 0pt, minimum size=2mm] at (1,0) {};
\node [label=center:{$\dots$}]at (2,0) {};
\node [fill,circle,draw,inner sep = 0pt, outer sep = 0pt, minimum size=2mm] at (3,0) {};
\node [fill,circle,draw,inner sep = 0pt, outer sep = 0pt, minimum size=2mm] at (4,0) {};
\node [label=center:$\vdots$] at (0,1.15) {};
\node [fill,circle,draw,inner sep = 0pt, outer sep = 0pt, minimum size=2mm] at (0,2) {};
\node [fill,circle,draw,inner sep = 0pt, outer sep = 0pt, minimum size=2mm] at (1,2) {};
\node [label=center:{$\dots$}]at (2,2) {};
\node [fill,circle,draw,inner sep = 0pt, outer sep = 0pt, minimum size=2mm] at (3,2) {};
\node [fill,circle,draw,inner sep = 0pt, outer sep = 0pt, minimum size=2mm] at (4,2) {};
\node [label=center:$\vdots$] at (4,1.15) {};
\node [label=center:$\iddots$] at (2,1.1) {};

\draw [decorate,decoration={brace,amplitude=10pt}]
(4.0,-0.5) -- (0.0,-0.5) node [black,midway,yshift=-0.6cm] {$m+1$};

\draw [decorate,decoration={brace,amplitude=10pt}]
(-0.5,-0.0) -- (-0.5,2.0) node [black,midway,xshift=-0.6cm, rotate=90] {$n+1$};
\end{tikzpicture}
\end{figure}

\vspace{5mm}

Then we can consider the rows and columns of these diagrams to obtain a family of diagrams for $G=C_{p^m}$, namely $\{X_i\}_{1 \leqslant i \leqslant n+1}$ and a family of diagrams for $G=C_{p^n}$, namely $\{Y_i\}_{1 \leqslant i \leqslant m+1}$. Note that these are indeed valid diagrams as can be seen from observing the restriction and composition rules.

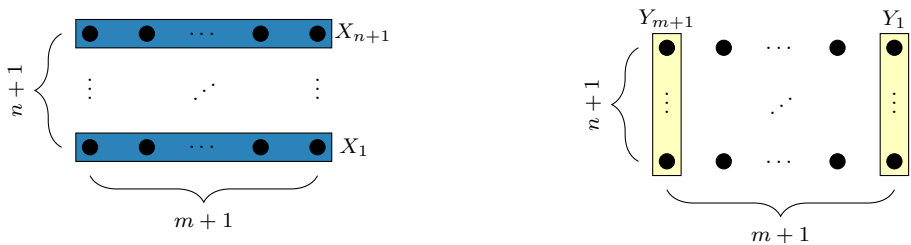
\begin{figure}[h!]
\centering
\begin{minipage}{.5\textwidth}
\centering
\begin{tikzpicture}[scale = 0.75]
\draw [black,fill=gGreen] (-.25,-0.25) rectangle (4.25,0.25);
\draw [black,fill=gGreen] (-.25,1.75) rectangle (4.25,2.25);

\node [fill,circle,draw,inner sep = 0pt, outer sep = 0pt, minimum size=2mm] at (0,0) {};
\node [fill,circle,draw,inner sep = 0pt, outer sep = 0pt, minimum size=2mm] at (1,0) {};
\node [label=center:{$\dots$}]at (2,0) {};
\node [fill,circle,draw,inner sep = 0pt, outer sep = 0pt, minimum size=2mm] at (3,0) {};
\node [fill,circle,draw,inner sep = 0pt, outer sep = 0pt, minimum size=2mm] at (4,0) {};
\node [label=center:$\vdots$] at (0,1.15) {};
\node [fill,circle,draw,inner sep = 0pt, outer sep = 0pt, minimum size=2mm] at (0,2) {};
\node [fill,circle,draw,inner sep = 0pt, outer sep = 0pt, minimum size=2mm] at (1,2) {};
\node [label=center:{$\dots$}]at (2,2) {};
\node [fill,circle,draw,inner sep = 0pt, outer sep = 0pt, minimum size=2mm] at (3,2) {};
\node [fill,circle,draw,inner sep = 0pt, outer sep = 0pt, minimum size=2mm] at (4,2) {};
\node [label=center:$\vdots$] at (4,1.15) {};
\node [label=center:$\iddots$] at (2,1.1) {};

\draw [decorate,decoration={brace,amplitude=10pt}]
(4.0,-0.5) -- (0.0,-0.5) node [black,midway,yshift=-0.6cm] {$m+1$};

\draw [decorate,decoration={brace,amplitude=10pt}]
(-0.5,-0.0) -- (-0.5,2.0) node [black,midway,xshift=-0.6cm, rotate=90] {$n+1$};

\node at (4.8,2.0) {$X_{n+1}$};
\node at (4.65,0.0) {$X_1$};
\end{tikzpicture}
\end{minipage}%
\begin{minipage}{.5\textwidth}
\centering
\begin{tikzpicture}[scale = 0.75]
\draw [black,fill=gYellow] (-.25,-0.25) rectangle (.25,2.25);
\draw [black,fill=gYellow] (3.75,-0.25) rectangle (4.25,2.25);

\node [fill,circle,draw,inner sep = 0pt, outer sep = 0pt, minimum size=2mm] at (0,0) {};
\node [fill,circle,draw,inner sep = 0pt, outer sep = 0pt, minimum size=2mm] at (1,0) {};
\node [label=center:{$\dots$}]at (2,0) {};
\node [fill,circle,draw,inner sep = 0pt, outer sep = 0pt, minimum size=2mm] at (3,0) {};
\node [fill,circle,draw,inner sep = 0pt, outer sep = 0pt, minimum size=2mm] at (4,0) {};
\node [label=center:$\vdots$] at (0,1.15) {};
\node [fill,circle,draw,inner sep = 0pt, outer sep = 0pt, minimum size=2mm] at (0,2) {};
\node [fill,circle,draw,inner sep = 0pt, outer sep = 0pt, minimum size=2mm] at (1,2) {};
\node [label=center:{$\dots$}]at (2,2) {};
\node [fill,circle,draw,inner sep = 0pt, outer sep = 0pt, minimum size=2mm] at (3,2) {};
\node [fill,circle,draw,inner sep = 0pt, outer sep = 0pt, minimum size=2mm] at (4,2) {};
\node [label=center:$\vdots$] at (4,1.15) {};
\node [label=center:$\iddots$] at (2,1.1) {};

\draw [decorate,decoration={brace,amplitude=10pt}]
(4.0,-0.5) -- (0.0,-0.5) node [black,midway,yshift=-0.6cm] {$m+1$};

\draw [decorate,decoration={brace,amplitude=10pt}]
(-0.5,-0.0) -- (-0.5,2.0) node [black,midway,xshift=-0.6cm, rotate=90] {$n+1$};

\node at (0.0,2.5) {$Y_{m+1}$};
\node at (4.0,2.5) {$Y_{1}$};
\end{tikzpicture}
\end{minipage}
\end{figure}

We shall say that an $N_\infty$-operad is \emph{pure} if it is completely determined by the systems $\{X_i\}$ and $\{Y_j\}$ in the sense of Remark~\ref{rem:comp}. 
If an operad is not pure, then we will say that it is \emph{mixed}.
Note that an operad is mixed if and only if after removing all norm maps of the form $N_{p^i}^{p^jq}$, $j \neq i$, and completing the set of norm maps according to the rules of Corollary \ref{operad-struc}, one does \emph{not} recover the original operad one started with.


\begin{example}
The following operad is pure as it has no diagonals, that is, no norm maps of the form $N_{p^i}^{p^jq}$. Therefore there is no condition to check


\begin{figure}[H]
 \begin{tikzpicture}[->, node distance=2cm, auto,scale = 0.65]
\node (11) at (2.000000,0) {$C_{p}$};
\node (p1) at (0.000000,0) {$C_{1}$};
\node (q1) at (2.000000,-2.0) {$C_{pq}$};
\node (pq1) at (0.000000,-2.0) {$C_{q}$};
\draw (p1) to (11);
\draw (p1) to (pq1);
\end{tikzpicture}
\end{figure}


The following is also pure, as when we remove the diagonal (highlighted in red) then the composition rule of Corollary~\ref{operad-struc} is violated. Completing the set of norm maps according to the rules forces the diagonal, and we recover the original operad that we started with.


\begin{figure}[H]
 \begin{tikzpicture}[->, node distance=2cm, auto,scale = 0.65]
\node (11) at (2.000000,0) {$C_{p}$};
\node (p1) at (0.000000,0) {$C_{1}$};
\node (q1) at (2.000000,-2.0) {$C_{pq}$};
\node (pq1) at (0.000000,-2.0) {$C_{q}$};
\draw[red] (p1) to (q1);
\draw (p1) to (11);
\draw (p1) to (pq1);
\draw (pq1) to (q1);
\end{tikzpicture}
\end{figure}
\end{example}


By using the restriction rules, we see that there is a natural ordering on the systems $\{X_i\}$ and $\{Y_j\}$. Indeed, $X_1 \leqslant X_2 \leqslant \cdots \leqslant X_{n+1}$ and $Y_1 \leqslant Y_2 \leqslant \cdots \leqslant Y_{m+1}$.

\begin{defn}
We will denote by $\mathcal{P}(n,r)$ the number of length $r$ paths in the $n$-Tamari lattice $\mathcal{A}_n$. For example, $\mathcal{P}(n,2)$ gives the sequence $1, 1, 3, 13, 68, 399, 2530, 16965, \dots$ (starting at $n=0$).
In Ch\^{a}tel and Pons~\cite{MR3345297} this is given the closed form
$$\dfrac{2(4n+1)!}{(n+1)!(3n+2)!}.$$
\end{defn}

\begin{thm}\label{thm2}
The number of pure $N_\infty$-operads for $G=C_{p^n}C_{q^m}$ is given as
$$\mathcal{P}(n+1,m) \mathcal{P}(m+1,n)$$
In general, for $G=C_{p^{n_1}_1} \cdots C_{p^{n_k}_k}$ the number of pure operads is
$$
\prod^k_{j,i=1} \mathcal{P}(n_i+1,n_j).
$$
\end{thm}

\begin{proof}
This is an exercise in counting using the orderings $X_1 \leqslant X_2 \leqslant \cdots \leqslant X_{n+1}$ and $Y_1 \leqslant Y_2 \leqslant \cdots \leqslant Y_{m+1}$. Once we have picked $X_1$, we must take a (possibly stationary) path of length $n$ through the Tamari lattice $\mathcal{A}_{m+1}$ to pick the other entries. Therefore, there are $\mathcal{P}(m+1,n)$ such options for the $X_i$. We then have the choices for the $Y_j$ giving us total of $\mathcal{P}(n+1,m)$ options via a similar argument. Combining these, we get the required total of $\mathcal{P}(n+1,m) \mathcal{P}(m+1,n)$.

The proof for the general case follows similarly.
\end{proof}

\begin{example}
One can compute the first few values for the sequence appearing in Theorem~\ref{thm2} (starting at $n=0$ for $m=1$) to be $1,9,52,340,2394,17710, \dots$. This sequence does not appear on the \texttt{OEIS} at the time of writing.
\end{example}

\bibliographystyle{plain}
\bibliography{norm}

\end{document}